\documentclass[10pt]{amsart}

\usepackage{amsmath,latexsym,amssymb,amsthm,enumerate,amsmath,amscd}
\usepackage[mathscr]{eucal}
\usepackage{color,fullpage,kotex}
\usepackage[usenames,dvipsnames]{xcolor}
\usepackage{colortbl}
\usepackage[all,cmtip]{xy}
\usepackage{graphicx,mathrsfs}
\usepackage{comment,youngtab,ytableau}
\usepackage{tabulary,longtable,tabu,multirow,multicol}
\usepackage{color}
\usepackage[colorlinks=true, linkcolor=blue, anchorcolor=blue, citecolor=blue, filecolor=blue, menucolor= blue, urlcolor=red]{hyperref}
\usepackage{tikz,url}
\usetikzlibrary{calc,shadings,patterns}
\usepackage{tikz-cd}

\theoremstyle{plain}
\newtheorem{theorem}{Theorem}[section]
\newtheorem{proposition}[theorem]{Proposition}
\newtheorem{corollary}[theorem]{Corollary}
\newtheorem{lemma}[theorem]{Lemma}
\newtheorem{thm}{Theorem}
\theoremstyle{definition}
\newtheorem{definition}[theorem]{Definition}

\newtheorem{example}[theorem]{Example}

\newtheorem{lem}[theorem]{Lemma}

\newtheorem{rem}[thm]{Remark}

\numberwithin{equation}{section}
\numberwithin{table}{section}
\definecolor{purple}{rgb}{0.4,0.2,0.4}

\def\<{\left<}
\def\>{\right>}

\definecolor{med-gray}{gray}{0.5}
\definecolor{gray1}{gray}{0.87}
\definecolor{gray2}{gray}{0.74}
\definecolor{gray3}{gray}{0.64}
\definecolor{gray4}{gray}{0.48}
\definecolor{verylight-yellow}{rgb}{1,1,0.7}
\definecolor{yellow}{rgb}{1,1,0.2}
\definecolor{vivid-blue}{rgb}{0.2,0,1}
\definecolor{light-pink}{rgb}{1,0.8,1}
\definecolor{med-pink}{rgb}{1,0.6,1}
\definecolor{aqua}{rgb}{0.0, 1.0, 1.0}
\definecolor{light-gray}{rgb}{0.5, 0.9, 0.5}

\def\char{{\rm{char}}\kern.03em}
\def\k{\Bbbk}

\def\sl{\mathfrak{sl}}

\allowdisplaybreaks

\begin{document}

\title{Representation Theory of Symmetric Groups \\ and the Strong Lefschetz Property}

\author[Seok-Jin Kang]{Seok-Jin Kang$^{*}$}
\address{Korea Research Institute of  Arts and Mathematics, Asan-si,
Chungcheongnam-do, 31551, Korea}
\email{soccerkang@hotmail.com}

\thanks{$^{*}$ This research was supported by Hankuk University of Foreign Studies Research Fund.}

\author[Y.R. KIM]{Young Rock Kim${}^{**}$}
\address{Major in Mathematics Education, Graduate School of Education, Hankuk University of Foreign Studies, Seoul,  02450, Korea}
\email{rocky777@hufs.ac.kr} %
\thanks{${}^{**}$ This research was supported by the Basic Science Research Program of the 
NRF (Korea) under grant No. \\ \indent ~~2015R1D1A1A01059643.}

\author[Y.S. Shin]{Yong-Su Shin${}^\dag$}
\thanks{${}^\dag$This research was supported by the Basic Science Research Program of the NRF (Korea) under grant No.\\ \indent ~~2019R1F1A1056934}
\address{Department of Mathematics, Sungshin Women's University,  Seoul, 02844, Korea and School of Mathematics, Korea Institute for Advanced Study, Seoul, 02455, Korea}
\email{ysshin@sungshin.ac.kr}


\begin{abstract} 

We investigate the structure and properties of an Artinian monomial complete intersection quotient 
$A(n,d)=\k [x_{1}, \ldots, x_{n}]   \big /   (x_{1}^{d}, \ldots, x_{n}^d)$. We construct explicit homogeneous bases of $A(n,d)$ that are compatible with the $S_{n}$-module structure for $n=3$,  all exponents $d \ge 3$ and all homogeneous degrees $j \ge 0$. Moreover, we derive the multiplicity formulas, both in recursive form and in closed form,  
for each irreducible component appearing in the $S_{3}$-module decomposition of homogeneous subspaces. 

\end{abstract}

\maketitle 


\section*{Introduction} 

\vskip 2mm 

The purpose of this paper is to investigate the structure and properties of an 
Artinian monomial complete intersection quotient
$$A(n,d) \cong \k[x_{1}, \ldots, x_{n}] \big / (x_{1}^{d}, \ldots, x_{n}^{d}) 
=\bigoplus_{j=0}^{m(d)} A(n,d)_{j},$$
where $\k$ is an algebraically closed field of characteristic $0$ and $m(d) = n(d-1)$. 

\vskip 2mm 

In \cite{St:1}, R. Stanley proved that $A(n,d)$ has the strong Lefschetz property in the narrow sense
with $\ell = x_{1} + \cdots + x_{n}$ as a strong Lefschetz element. That is, the linear map 
$F^{d} = \times \ell^{d}: \, A_{j} \rightarrow A_{j+d}$ has the maximal rank for all $j\ge 0$, $d\ge 1$
and its Hilbert polynomial 
$$\text{Hilb}(A(n,d), t):=\sum_{j=0}^{m(d)} (\dim A(n,d)_{j}) t^{j}$$
satisfies condition
$$\dim A(n,d)_{j} = \dim A(n,d)_{m(d) -j} \quad \text{for all} \ \ j=0, 1, \ldots, \left\lfloor \frac{m(d)}{2} \right\rfloor.$$ 
Since it is a major breakthrough in the theory of Lefschetz properties, Stanley's Theorem 
has been reproved in various contexts. In particular, J. Watanabe took a representation-theoretic 
approach to this theorem, which attracts our attention. 

\vskip 2mm 

The symmetric group $S_{n}$ acts on the polynomial algebra $\k[x_{1}, \ldots, x_{n}]$ by permuting the 
variables, which induces a natural $S_{n}$-action on $A(n,d)$. Since $\ell=x_{1} + \cdots + x_{n}$ is 
invariant under the $S_{n}$-action, the linear map $F=\times \ell : A(n,d)_{j} \rightarrow A(n,d)_{j+1}$ 
commutes with the symmetric group action. Moreover, the linear map $E: A(n,d)_{j+1} \rightarrow
A(n,d)_{j}$, defined by 
$$E(x_{1}^{a_1} \cdots x_{n}^{a_n}) = \sum_{k=0}^{n} a_{k} (d-a_{k}) x_{1}^{a_1} \cdots 
x_{k}^{a_k -1} \cdots x_{n}^{a_n},$$
also commutes with the $S_{n}$-action.
It is straightforward to verify that the linear maps $F$, $E$, $H:=[E, F]$ generate a Lie algebra isomorphic to
$\sl_{2}$. Hence the the algebra $A(n,d)$ has a $(\k[S_{n}] \times U(\sl_{2}))$-bimodule structure,
where $\k[S_{n}]$ is the group algebra of $S_{n}$ and $U(\sl_{2})$ is the universal enveloping algebra of $\sl_{2}$.  

\vskip 2mm

One of our main goals of this paper is to give an explicit construction of homogeneous bases of 
the algebra $A(n,d)$ which is compatible with the $S_{n}$-module structure. As we have seen 
in the above discussion, this problem is reduced to finding such bases of $\text{Ker}(E) \cap A(n,d)_{j}$
for $j=0, 1, \ldots, \lfloor \frac{m(d)}{2} \rfloor$. We achieve our goal by combining the representation theory
of symmentric groups and the standard $\sl_{2}$-theory. In addition, we determine the multiplicity of 
each irreducible $S_{n}$-module appearing in the decomposition of homogeneous subspaces
$\text{Ker}(E)$. In Section 3, we express these multiplicities in terms 
of {\it rectangular partitions} and in Section 4, we prove interesting  {\it recusrsive relations} between the multiplicities.
Finally, in Section 5, we derive explicit {\it closed form formulas} for the multiplicities.
In this paper, we focus on the case when $n=3$, all the exponents $d \ge 3$ and all the  homogeneous degrees $j \ge 0$. The more general cases will be dealt with in our forthcoming paper \cite{KKS2}.  

\vskip 2mm 

This paper is organized as follows. In Section 1, we briefly review the standard  $\sl_2$-theory and explain 
the representation-theoretic approach to the Artinian graded $\k$-algebras having the strong Lefschetz 
property in the narrow sense.  
In Section 2, we recall some of basic representation theory of symmetric groups 
and recollect the related combinatorics of Young diagrams and Young tableaux. 
In Section 3, we explain the strategy to construct explicit homogeneous basis polynomials
and determine the multiplicities of irreducible submodules in terms of rectangular partitions. 
In Section 4 and Section 5, we derive the multiplicity formulas, both in recursive form and in closed form,  
for each irreducible component appearing in the $S_{3}$-module decomposition of homogeneous subspaces.

\vskip 5mm

\section{Basic $\sl_{2}$-theory and Strong Lefschetz property}

We begin with a brief review of basic $\sl_{2}$-theory and the properties of graded algebras with strong Lefschetz property in the narrow sense. In this paper, $\k$ denotes an algebraically closed field of characteristic $0$.  

Recall that $\sl_{2}$ is the Lie algebra generated by the elements $e$, $f$, $h$ with defining relations
 \begin{equation} \label{eq:sl2}
 [e, f] = h, \quad [h, e] = 2e, \quad [h,f] = -2f.
 \end{equation}

For each $m \in \mathbf{Z}_{\ge 0}$, it is well-known that there exists a unique (up to isomorphism) $(m+1)$-dimensional irreducible $\sl_{2}$-modules $V(m)$ with a basis $\{v_{0}, v_{1}, \ldots, v_{m} \}$, where the $\sl_{2}$-action is given by 
\begin{equation} \label{eq:V(m)}
\begin{aligned}
& h \cdot v_{k} = (m-2k)\, v_{k}, \\
& f \cdot v_{k} = v_{k+1}, \\
& e \cdot v_{k} = k (m_k+1) \,v_{k-1}. 
\end{aligned}
\end{equation}
Here, we understand $v_{-1} = v_{m+1} =0$. 

\vskip 2mm 

Let $V$ be a finite-dimensional $\sl_{2}$-module. We say that a non-zero vector $v \in V$ {\it has weight $\lambda$} if $h \cdot v = \lambda \, v$ for some $\lambda \in \k$. Moreover, $v$ is called a {\it highest weight vector } (resp. {\it lowest weight vector})  if $e \cdot v=0$ (resp. $f \cdot v =0$). Thus, in $V(m)$, $v_{k}$ has weight $m-2k$, $v_{0}$ is a highest weight vector with weight $m$ and $v_{m}$ is a lowest weight vector with weight $-m$. 

\vskip 2mm 

Over an algebraically closed field of characteristic $0$, {\it Weyl's Theorem} implies that every finite-dimensional 
$\sl_{2}$-module is completely reducible (see, for example, \cite{Humph78}). Let $V$ and $W$ be 
finite-dimensional $\sl_{2}$-modules. Then $V \otimes W$ becomes an $\sl_{2}$-module via
\begin{equation*}
x \cdot (v \otimes w) = x \cdot v \otimes w + v \otimes x \cdot w \quad \text{for} \ x\in \sl_{2}, \ v \in V, \ w\in W.
\end{equation*}

For $m \ge n \ge 0$, by Weyl's Theorm, the tensor product $V(m) \otimes V(n)$ is completely reducible and 
its irreducible decomposition is given by {\it Clebsch-Gordan formula}:
\begin{equation} \label{eq:CG}
V(m) \otimes V(n) \cong V(m+n) \oplus V(m+n-2) \oplus \cdots \oplus V(m-n).
\end{equation}

For example, we have
\begin{equation*}
\begin{aligned}
& V(3) \otimes V(2) \cong V(5) \oplus V(3) \oplus V(1), \\
& V(2) \otimes V(2) \otimes V(2) \cong V(6) \oplus V(4)^{\oplus 2} \oplus V(2)^{\oplus 3} \oplus V(0).
\end{aligned}
\end{equation*}

\vskip 2mm 

Let $A=\bigoplus_{j=0}^{m} A_{j}$ be a finite-dimensional graded $\k$-algebra. 
We define its {\it Hilbert polynomial} by 
$$\text{Hilb}(A, t):= \sum_{j=0}^{m} (\dim\, A_{j})\, t^{j}.$$

We say that $A$ has the {\it strong Lefschetz property in the narrow sense} if there exists an element 
$\ell  \in A_{1}$ such that the linear map $F^{d} = \times \ell^{d} : A_{j} \rightarrow A_{j+d}$ has maximal rank 
for all $j \ge 0$, $d \ge 1$ and  
$\dim A_{j} = \dim A_{m-j}$ for $0 \le j \le \lfloor \frac{m}{2} \rfloor$. In this case, $\ell$ is called a 
{\it strong Lefschetz  element}. 

\vskip 2mm 

Since $A$ is finite-dimensional, the linear map $F=\times \ell : A \rightarrow A$ is nilpotent. Consider
the Jordan canonical form of $F$ with Jordan blocks of size $m_{1}, \ldots, m_{r}$. We rearrange the
Jordan blocks so that we have $m_{1} \ge m_{2} \ge \cdots \ge m_{r} >0$. For each $j=1, \ldots, r$, let 
$B_{j}=\{u_{j}, F(u_{j}), \ldots, F^{m_{j}-1}(u_{j}) \}$ be a Jordan canonical basis of the corresponding 
Jordan block. By our construction, $B:=\sqcup_{j=1}^{r} B_{j}$ is a Jordan canonical basis of $A$ 
with respect to $F$. 

\vskip 2mm 

For each $j=1, \ldots, r$, let $A^{(j)}$ be the subspace of $A$ spanned by $B_{j}$ and define the linear maps
$H: A^{(j)} \rightarrow A^{(j)}$, \ $E: A^{(j)} \rightarrow A^{(j)}$ by 
\begin{equation*}
H(F^{k}(u_{j})) = (m_{j} -2k)\, F^{k}(u_{j}), \quad E(F^{k}(u_{j})) = k(m_{j} - k) F^{k-1}(u_{j}) \quad (0 \le k \le m_{j}-1).
\end{equation*}
Then it is straightforward to verify that $E$, $F$, $H$ satisfy the defining relation \eqref{eq:sl2} and for each 
$j=1, \ldots, r$, the vectors $F^{k}(u_{j})$ satisfy the relations in  \eqref{eq:V(m)}. It follows that each 
subspace $A^{(j)}$ is isomorphic to the irreducible $\sl_{2}$-module $V(m_{j}-1)$. 

\vskip 2mm 

Hence $A$ becomes a completely reducible $\sl_{2}$-module and its irreducible decomposition 
is determined by the Jordan canonical form of $F$. More precisely, we have 
\begin{equation} \label{eq:sl2decomp}
A \cong \bigoplus_{j=0}^{\lfloor \frac{m}{2} \rfloor} V(m-2j)^{\oplus a_{j}},
\end{equation}
where $a_{0}=1$, $a_{j} = \dim A_{j} - \dim A_{j-1}$ for $0 \le j \le \lfloor \frac{m}{2} \rfloor$. 

\vskip 2mm 

\begin{example}\label{ex:Hilb} 

Let $A=\k[x_1,x_2,x_3] \big / (x_{1}^3, x_{2}^3, x_{3}^3)$. Then its Hilbert polynomial is 
$$\text{Hilb}(A, t) = 1 + 3t + 6t^2 + 7t^3 + 6t^4 + 3t^5 + t^6.$$
Hence the $\sl_2$-decomposition of $A$ is given by 
$$A \cong V(6) \oplus V(4)^{\oplus 2} \oplus V(2)^{\oplus 3} \oplus V(0).$$
Note that $A \cong V(2)^{\otimes 3}$ as an $\sl_{2}$-module. 
\end{example}

\vskip 3mm 

The following proposition is one of the main results in \cite{H-W2080}. The key ingredient of its proof 
is the Clebsch-Gordan formula.

\begin{proposition}[\cite{H-W2080}]
Let $(A_{1}, \ell_{1})$ and $(A_{2}, \ell_{2})$ be  finite dimensional graded algebras having the strong  
Lefschetz property in the narrow sense  with strong Lefschetz elements $\ell_1$ and $\ell_2$, respectively. 
Then  $A_{1} \otimes A_{2}$ also has the strong Lefschetz property in the narrow sense with a strong Lefschetz element $\ell=\ell_{1} \otimes 1 + 1 \otimes \ell_{2}$. 
\end{proposition}

As an immediate corollary, we obtain a representation-theoretic proof of Stanley's Theorem.  

\begin{corollary}[\cite{H-W2080, Wa:1}]
An Artinian monomial complete intersection quotient 
$$A=\k[x_{1}, \ldots, x_{n}] \big / (x_{1}^{d_1}, \ldots, x_{n}^{d_{n}})$$
has the strong Lefschetz property in the narrow sense with a strong Lefschetz element $\ell=x_{1} + \cdots + x_{n}$. 
\end{corollary}

\vskip 5mm

\section{Representation theory of symmetric groups}

We now recollect some pieces of the representation theory of $S_{n}$. A weakly decreasing sequence of positive integers $\lambda = (\lambda_{1} \ge \lambda_{2} \ge \cdots \ge \lambda_{r} >0)$ is called a {\it partition} of $n$ if 
$\lambda_1 + \lambda_2 + \cdots + \lambda_r = n$. The positive integer $r$ is called the {\it length} of $\lambda$, usually written as $\ell(\lambda)$. We denote by $\mathscr{P}(n)$ the set of all partitions of $n$ and $p(n)$ the number of elements in $\mathscr{P}(n)$. 
A partition $\lambda = (\lambda_1 \ge \lambda_2 \ge \cdots \ge \lambda_r>0)$ is called a {\it rectangular
partition of {siz} $k \times l$} 
if $\ell(\lambda) \le k$ and all the entries in $\lambda$ are $\le l$.
We denote by ${\mathscr P}_{k,l}(n)$ the set of all rectangular partitions of $n$ of size $k \times l$ and $p_{k,l}(n)$ the 
number of elements in $\mathscr{P}(n)$.  

\vskip 2mm 

We identify a partition $\lambda=(\lambda_1 \ge \lambda_2 \ge \cdots \ge \lambda_r>0)$ with the {\it Young diagram} $Y^{\lambda}$ which consists of $\lambda_1$-many boxes in the first row, $\lambda_2$-many boxes in the second row, etc. For instance, the partition $\lambda=(5,3,2,2)$ of $12$ is identified with the following Young diagram.
\vskip 1 true mm
$$
Y^{\lambda}= \ydiagram{5,3,2,2}
$$
\vskip 1.5 true mm

\vskip 2mm 

\begin{definition} 

Let $\lambda$ be a partition of $n$. 
A {\it standard tableau of shape $\lambda$} is a filling of the boxes of $Y^{\lambda}$ with entries taken from $1, 2, \ldots, n$ such that 
\begin{enumerate}
\item[(i)] all the entries in each row are strictly increasing from left to right, 

\item[(ii)] all the entries in each column are strictly increasing from top to bottom.
\end{enumerate}
\end{definition}

An example of standard tableau of shape $\lambda=(5,3,2,2)$ is given below.

\vskip 2mm 
$$
T=\begin{ytableau} 1&2&4&6& 9 \\3&8&11\\5&10\\7&12 \end{ytableau}
$$
%
%
%
%

\vskip 3mm 

To each partition $\lambda$ of $n$, there exists a unique (up to isomorphism) irreducible $S_{n}$-module $S^{\lambda}$ which is called the {\it Specht module} associated with $\lambda$. It is known that every finite dimensional irreducible $S_{n}$-module has the form  $S^{\lambda}$ for some partition $\lambda$ of $n$.  Furthermore, the dimension of $S^{\lambda}$ is given by the number of standard tableaux of shape $\lambda$. (See, for example, \cite{Fulton97, Fulton-Harris, Sagan2001}). 

\vskip 2mm 

There are three special types of irreducible $S_{n}$-modules: {\it trivial representation}, {\it sign representation} and {\it 
standard representation}. 

\vskip 2mm 

\begin{enumerate}

\item[(a)] The trivial representation is the 1-dimensional $S_{n}$-modules on which every permutaton acts as the identity. One can see that the trivial representation corresponds to the partition $\lambda = (n)$.

\item[(b)] The sign representation is the 1-dimensional $S_{n}$-module on which every permutation acts as the multiplication by its sign. The sign representation corresponds to the partition $\lambda = (1, 1, \ldots, 1)$.

\item[(c)] The standard module is the Specht module corresponding to the partition $\lambda =  (n-1, 1)$. It can be identified with the $(n-1)$-dimensional hyperplane $H=\{(x_1, \ldots, x_n) \in \k^{n}  \mid x_{1} + \cdots + x_{n} = 0 \}$. 

\end{enumerate}

\vskip 2mm 

When $n=3$, they are all ireducible representations of the symmetric group.  

\vskip 3mm 

\begin{example} In the following figure, we exhibit all the standard tableaux corresponding to the partitions $(3)$, $(1,1,1)$ and $(2,1)$.

$$
\begin{array}{lllllllllll}
\lambda =(3), & \begin{ytableau} 1&2&3\end{ytableau}\, , \\[2.0ex] 
 \lambda =(1,1,1), & \begin{ytableau} 1\\2\\3\end{ytableau}\, ,\\[8.5ex]
 \lambda = (2,1), & \begin{ytableau} 1&2\\3\end{ytableau}, \quad \begin{ytableau} 1&3\\2\end{ytableau} \, .

\end{array}
$$
Hence $\text{dim} S^{(3)} = \text{dim} S^{(1,1,1)} =1$ and $\text{dim} S^{(2,1)}=2$. 
\end{example}

\vskip 5mm

\section{Artinian monomial complete intersection quotients} 

Let $A:= A(n,d)= \k[x_{1}, \ldots, x_{n}] \big / (x_{1}^d, \ldots, x_{n}^d) = A_{0} \oplus A_{1} 
\oplus \cdots \oplus A_{m}$ be an Artinian monomial complete  intersection quotient with 
the same exponent $d$, where $m=n(d-1)$. 
Recall that the symmetric group $S_{n}$ acts on $\k[x_{1}, \ldots, x_{n}]$ by permuting the variables. Since the generators have the same exponents, we have a natural $S_{n}$-action on $A$.  
We have seen that $A$ has the  strong Lefschetz property in the narrow sense with a strong Lefschetz element 
$\ell=x_{1} + \cdots + x_{n}$. Define a linear map $E: A_{j+1} \rightarrow A_{j}$ by 
\begin{equation} \label{eq:E}
E(x_{1}^{a_{1}} \cdots x_{n}^{a_{n}}) = \sum_{k=1}^{n} a_{k} (d-a_{k}) x_{1}^{a_{1}} \cdots x_{k}^{a_{k}-1} \cdots x_{n}^{a_{n}}.
\end{equation}
Then the linear maps  $E$, $F=\times \ell$, $H:=[E, F]$ generate a Lie algebra which is isomorphic to $\sl_2$. 

\vskip 2mm

Clearly, each homogeneous subspace $A_{j}$ is invariant under the $S_{n}$-action on $A$. Since $\ell=x_{1} + \cdots + x_{n}$ is $S_{n}$-invariant, we have $F \circ \sigma = \sigma \circ F$ for all $\sigma \in S_{n}$. Furthermore, it can be shown that $E \circ \sigma = \sigma \circ E$ for all $\sigma \in S_{n}$ (cf. \cite{H-W2080}). That is, $E$ and $F$ are $S_{n}$-module homomorphisms.

\vskip 2mm 

One of the main goals of this paper is to construct explicit homogeneous bases of irreducible
$S_{n}$-submodules that appear as a direct summand in the $S_{n}$-module decomposition of $A$. 
Since $A$ is a completely reducible $\sl_{2}$-module, by Schur's Lemma, it suffices to construct
such bases for $\text{Ker}(E) \cap A_{j}$ $(j=0, 1, \ldots, \lfloor \frac{m}{2} \rfloor)$. 
For general $n$, $d$ and $j$, it is not yet known how to construct such an explicit basis directly. 
(In \cite{Maeno}, Maeno gave an inductive algorithm.) 
We expect the general case can be solved using the Young tabloid construction of 
Specht modules \cite{KKS2}. 

\vskip 2mm 

In this paper, we focus on the case when $n=3$ and deal with all the exponents $d \ge 3$ and all the homogeneous degrees $j \ge 0$. Recall that all the irreducible $S_{3}$-modules are classified as one of the following\,:
trivial representation, sign representaion and standard representation. We give an explicit construction of 
homogeneous bases for these irreducible $S_{3}$-submodules of $\text{Ker}(E) \cap A_{j}$ 
$(j=0, 1, \ldots, \lfloor \frac{m}{2} \rfloor)$. 
Moreover, we determine their multiplicities in the irreducible decomposition of $\text{Ker}(E)$ in each homogeneous subspace of $A$. 

\vskip 2mm 

For simplicity, we will write $A(d)$ for $A(n,d)$ and denote 
$$\text{mult}(d, j) = \dim (\text{Ker}(E) \cap A(d)_{j}) = \dim A(d)_{j} - \dim A(d)_{j-1},$$ 
the multiplicity of the irreducible highest weight $\sl_2$-module $V(3(d-1)-2j)$ in $A(d)_{j}$. 
Then using the Hilbert polynomial of $A(d)$, we obatin
\begin{equation} \label{eq:mult}
\text{mult}(d, j) = \begin{cases}
j+1  & \text{for} \ \ 0 \le j \le d-1, \\
3d - 2 - 2j & \text{for} \ \ d \le j \le \left \lfloor \dfrac{3(d-1)}{2} \right \rfloor.
\end{cases}
\end{equation}

\vskip 3mm

Now we state and prove our main results. 

\vskip 2mm 

{\bf (1) Trivial representation: } 

\vskip 2mm 

Given $d$ and $j$, since $S_{3}$ acts trivially, any basis polynomial of the trivial representation in $A(d)_{j}$ is a homogeneous symmetric polynomial of degree $j$, say, $P$. Hence it can be expressed as 
$$P=\sum_{\lambda = (a, b, c)} \alpha_{a, b, c} [x_{1}^{a}  x_{2}^{b} x_{3}^{c}],$$
where $\lambda = (a,b,c)$ is a partition of $j$ of size $3 \times (d-1)$ and  $[x_{1}^{a}  x_{2}^{b} x_{3}^{c}]$ denotes the $S_{3}$-orbit of the monomial $x_{1}^{a} x_{2}^{b} x_{3}^{c}$ (under the permutation action of variables). For instance, when $d \ge 4$, we have 
\begin{equation*}
\begin{aligned}
& [x_{1}^3] = x_{1}^{3} +x_{2}^{3} + x_{3}^{3}, \\
& [x_{1}^{2} x_{2}] = x_{1}^{2} x_{2} + x_{1}^{2} x_{3} + x_{2}^{2} x_{3} + x_{1} x_{2}^{2} + x_{1} x_{3}^{2} + x_{2} x_{3}^{2}, \\
& [x_{1} x_{2} x_{3}] = x_{1} x_{2} x_{3}.
\end{aligned}
\end{equation*}

Then the condition $E(P)=0$ would yield a system of linear equatons with variables $\alpha_{a,b,c}$. The solutions of this system give the desired explicit basis polynomials and the nullity of  the coefficient matrix is the multiplicity of the trivial representation in $\text{Ker}(E) \cap A(d)_{j}$. Note that the number of variables is the same as $p_{3, d-1}(j)$ and the number of equations is $p_{3, d-1}(j-1)$. Indeed, let ${\mathscr P}_{3,d-1}(j)=\{\lambda \mid \lambda=(a, b,c), d> a\ge b\ge c\ge 0\}$ with $a+b+c=j$. So the number of variables in $E(P)=0$ is $p_{3, d-1}(j)=\mid {\mathscr P}_{3,d-1}(j)\mid$. Moreover, notice that $\deg E(P)=j-1$. Let ${\mathscr P}_{3,d-1}(j)=\{\lambda_1,\dots,\lambda_s\}$ and let $\mu_1=(1,0,0),\mu_2=(0,1,0),\mu_3=(0,0,1)$. Then the collection of all partitions of the form $\lambda_i-\mu_j$ is ${\mathscr P}_{3,d-1}(j-1)$ (here $\lambda_i-\mu_j$ is not necessarily a partition). In other words, the number of equations from $E(P)=0$ is $p_{3,d-1}(j-1)=\mid {\mathscr P}_{3,d-1}(j-1)\mid$.


\begin{example} Let $n=3$, $d=4$, and $j=3$. Then ${\mathscr P}_{3,3}(3)=\{(3,0,0), (2,1,0), (1,1,1)\}$ and let
$$
\begin{array}{lllllllll}
P
& = & \alpha_{3,0,0}[x_1^3]+\alpha_{2,1,0}[x_1^2x_2]+\alpha_{1,1,1}[x_1x_2x_3] \\
& = & \alpha_{3,0,0}(x_1^3+x_2^3+x_3^3)+\alpha_{2,1,0}(x_1^2x_2+x_1^2x_3+x_2^2x_3+x_1x_2^2+x_1x_3^2+x_2x_3^2)+\alpha_{1,1,1}(x_1x_2x_3).
\end{array} 
$$
So $p_{3,3}(3)=\mid\{(3,0,0),(2,1,0),(1,1,1)\}\mid=3$ unknowns. Furthermore, 
\begin{align*}
&E(P)=3\alpha_{3,0,0}[x_1^2]+4\alpha_{2,1,0}[x_1x_2]+3\alpha_{2,1,0}[x_1^2]+3\alpha_{2,1,0}[x_1x_2]\\
&\phantom{E(P)}=(3\alpha_{3,0,0}+3\alpha_{2,1,0})[x_1^2]+ (4\alpha_{2,1,0}+3\alpha_{1,1,1})[x_1x_2]\\
&\phantom{E(P)}=  0
\end{align*} 
yields two equations, which is $p_{3,3}(2)=\mid \{(2,0,0),(1,1,0)\}\mid$.
\end{example}

The following lemma is easy to prove, but we introduce the proof for completeness.

\begin{lemma} \label{L:20190917-302} 
With notation as above, let $\mu=(a,b,c)\in {\mathscr P}_{3, d-1}(j-1)$. Then at least one of $\mu+(1,0,0), \mu+(0,1,0), \mu+(0,0,1)$ is a partition in ${\mathscr P}_{3,d-1}(j)$. 
\end{lemma}

\begin{proof} First note that since $j\le \Big\lfloor \frac{3(d-1)}{2}\Big\rfloor <3(d-1)$, we get that $c<d-1$. If $c<b$, then $\mu+(0,0,1)=(a,b,c+1)$ is a partition in ${\mathscr P}_{3,d-1}(j)$. If $b=c<a$, then $\mu+(0,1,0)=(a,b+1,c)$ is a partition in ${\mathscr P}_{3,d-1}(j)$. If $a=b=c<d-1$, then $\mu+(1,0,0)=(a+1,b,c)$ is a partition in ${\mathscr P}_{3,d-1}(j)$ as well. 
\end{proof} 

For $\lambda =(a, b, c)$ with $d>a\ge b\ge c\ge 0$, we denote $x^{\lambda} = x_{1}^{a} x_{2}^{b} x_{3}^{c}$ and write 
$$
E([x^\lambda])=E([x_1^a x_2^b x_3^c]) := C_{\mu_1,\lambda} [x_1^{a-1} x_2^b x_3^c]+C_{\mu_2,\lambda} [x_1^{a} x_2^{b-1} x_3^c]+
C_{\mu_3,\lambda} [x_1^{a} x_2^b x_3^{c-1}],
$$
where $\mu_1=(a-1,b,c), \mu_2=(a,b-1,c)$, and $\mu_3=(a,b,c-1)$. Note that if $\mu_i$ is not a partition, then
we have $C_{\mu_i,\lambda}=0$. 

\medskip

In general, let ${\mathscr P}_{3, d-1}(j)=\{\lambda_{1}, \ldots, \lambda_{t}\}$ be the collection of partitions of $j$ and let ${\mathscr P}_{3, d-1}(j-1)=\{\mu_{1}, \ldots, \mu_{s}\}$ be the collection of partitions of $j-1$ listed by the lexicographic ordering. Thus $t=p_{3, d-1}(j)$ and $s=p_{3, d-1}(j-1)$.   
Let 
$P=\alpha_{1} [x^{\lambda_{1}}] + \cdots + \alpha_{t} [x^{\lambda_{t}}]$ be a homogeneous symmetric polynomial of degree $j$. Then the condition $E(P)=0$ yields a system of linear equations in the variable $\alpha_{1}, \ldots, \alpha_{t}$ with coefficient matrix $C=(C_{pq})$ for $(1\le p \le s,1\le q \le t)$, where $C_{pq}:=C_{\mu_p,\lambda_q}$ such that 
$$
E\left(\alpha_q \left[x^{\lambda_q}\right]\right)=
\sum_{p=1} ^s \alpha_q C_{\mu_p,\lambda_q}\left[x^{\mu_p}\right]:=\sum_{p=1} ^s \alpha_q C_{pq} \left[x^{\mu_p}\right].
$$ 
Observe that the coefficients $C_{pq}=0$ unless $\lambda_q=\mu_{p} + (1,0,0)$, $\lambda_q=\mu_{p} + (0,1,0)$, or $\lambda_q=\mu_{p} + (0,0,1)$ and they are partitions. In other words, $C_{pq}=0$ unless the term $[x^{\mu_{p}}]$ arises as an image of $[x^{\mu_{p}+(1,0,0)}]$, $[x^{\mu_{p}+(0, 1,0)}]$, or $[x^{\mu_{p}+(0,0,1)}]$ under $E$. Also note that at least one of them is a partition.

\medskip

From this observation, we have
$$
\begin{array}{llccccc}
&C_{\mu_1,\lambda_1} \alpha_1[x^{\mu_1}]+\cdots+C_{\mu_s,\lambda_1} \alpha_1[x^{\mu_s}]+\\
&C_{\mu_1,\lambda_2} \alpha_2[x^{\mu_1}]+\cdots+C_{\mu_s,\lambda_2} \alpha_2[x^{\mu_s}]+\\
&\hskip 2.66 true cm \vdots \\
&C_{\mu_1,\lambda_t} \alpha_t[x^{\mu_1}]+\cdots+C_{\mu_s,\lambda_t} \alpha_t [x^{\mu_s}]\phantom{+}\\[1ex] 
= & (C_{\mu_1,\lambda_1} \alpha_1+C_{\mu_1,\lambda_2} \alpha_2+\cdots+C_{\mu_1,\lambda_t} \alpha_t)[x^{\mu_1}]+\\
   & (C_{\mu_2,\lambda_1} \alpha_1+C_{\mu_2,\lambda_2} \alpha_2+\cdots+C_{\mu_2,\lambda_t} \alpha_t)[x^{\mu_2}]+\\
   & \hskip 2.66 true cm \vdots \\
   & (C_{\mu_{s-1},\lambda_1} \alpha_1+C_{\mu_{s-1},\lambda_2} \alpha_2+\cdots+C_{\mu_{s-1},\lambda_t} \alpha_t)[x^{\mu_{s-1}}]+\\
   & (C_{\mu_s,\lambda_1} \alpha_1+C_{\mu_s,\lambda_2} \alpha_2+\cdots+C_{\mu_s,\lambda_t} \alpha_t)[x^{\mu_s}]\\
= & 0.   
\end{array} 
$$
Hence we obtain the following linear system 
\begin{equation}\label{EQ:20190918-303}
\begin{aligned}
& C_{11}  \alpha_{1} + \cdots + C_{1t}  \alpha_{t}=0, \\
& C_{21}  \alpha_{1} + \cdots + C_{2t}  \alpha_{t} =0, \\
& \hskip 15mm  \vdots \\
& C_{s1}  \alpha_{1} + \cdots + C_{st}  \alpha_{t} =0.
\end{aligned}
\end{equation}
Let $C(a)$ be the submatrix of $C$ such that i) the non-zero entries are determined by the partitions $\mu \in {\mathscr P}_{3, d-1}(j-1)$, $\lambda \in {\mathscr P}_{3, d-1}(j)$, where the first component of $\mu$ is $a$ and the first component of $\lambda$ is $a$ or $a+1$, ii) the other entries are $0$.

\vskip 3mm 

We now consider the following example.

\begin{example}\label{example3.2}  (a) Let $d=3$, $j=3$. Then $
{\mathscr P}_{3, d-1}(3)=\{\lambda_1=(2,1,0), \lambda_2=(1,1,1)\}$ 
and ${\mathscr P}_{3,d-1}(2)=\{\mu_1=(2,0,0), \mu_2=(1,1,0)\}$. 
The basis element of the trivial representation $P$ can be written as 
$$
\begin{array}{llllllllllll} 
P= \alpha_1[x^{\lambda_1}]+\alpha_2[x^{\lambda_2}]=\alpha_{1} [x_{1}^{2} x_{2}] + \alpha_{2} [x_{1} x_{2} x_{3}].
\end{array} 
$$
Hence we have
$$
\begin{array}{llllllll}
E([x^{\lambda_1}])
& = & E(x_1^2x_2+x_1^2x_3+x_1x_2^2+ x_1x_3^2+x_2^2x_3+x_2x_3^2) \\
& = & 2 x_1x_2+2x_1^2+
          2 x_1x_3+2x_1^2+ 
          2 x_2^2+2x_1x_2+  
          2 x_3^2+2x_1x_3+
          2 x_2x_3+2x_2^2+
          2 x_3^2+2x_2x_3\\
& = & 4[x^{\mu_1}]+4[x^{\mu_2}], \quad \text{and} \\[1ex] 
E([x^{\lambda_2}])
& = & E(x_1x_2x_3) \\
& = & 2 x_1x_2+ 2 x_1x_3+2x_2x_3\\
& = & 2[x^{\mu_2}].
\end{array} 
$$
Therefore we get
$$
\begin{array}{lllllll} 
C_{\mu_1,\lambda_1}=C_{11}=4, \\
C_{\mu_1,\lambda_2}=C_{12}=0, \\
C_{\mu_2,\lambda_1}=C_{21}=4, \\
C_{\mu_2,\lambda_2}=C_{22}=2.
\end{array} 
$$
This means that
$$
C=\begin{pmatrix}
C_{11} & C_{12} \\
C_{21} & C_{22} 
\end{pmatrix} 
=\begin{pmatrix}
4  & 0 \\
4  & 2  
\end{pmatrix}.
$$
In other words,  the linear system
$$
C\begin{pmatrix}
a_1 \\
a_2  
\end{pmatrix}=0
$$
has only the trivial solution. Hence there is no trivial representation in $\text{Ker}(E) \cap A(3)_{3}$. Note that $p_{3,3}(3) - p_{3,3}(2) = 2 - 2 =0$.

\vskip 2 true mm

Now we have
$$
C=
\begin{pmatrix}
C_{11}&C_{12}\\
C_{21}&C_{22}
\end{pmatrix}
=
\begin{pmatrix}
C_{\mu_1,\lambda_1}&C_{\mu_1,\lambda_2}\\
C_{\mu_2,\lambda_1}&C_{\mu_2,\lambda_2}
\end{pmatrix}
=
\begin{pmatrix}
C_{(2,0,0),(2,1,0)}&0\\
C_{(1,1,0),(2,1,0)}&C_{(1,1,0),(1,1,1)}
\end{pmatrix}
=
\begin{pmatrix}
C(2)\\
C(1)
\end{pmatrix}
$$
and then the rank of $C$ is equal to $2$ which is a sum of the number of rows of $C(a), a=2,1$. 

\vskip 2mm 

(b) Let $d=7$, $j=6$. Then we have 
$$ 
{\mathscr P}_{3, d-1}(6)=\{\lambda_1=(6,0,0), \lambda_2=(5,1,0), \lambda_3=(4,2,0), \lambda_4=(4,1,1), \lambda_5=(3,3,0),\lambda_6=(3,2,1),\lambda_7=(2,2,2)\}
$$
and 
$$
{\mathscr P}_{3, d-1}(5)=\{\mu_1=(5,0,0), \mu_2=(4,1,0), \mu_3=(3,2,0), \mu_4=(3,1,1), \mu_5=(2,2,1)\}.
$$
Thus the basis polynomial $P$ can be written as 
$$
P= \alpha_{1} [x_{1}^{6}] + \alpha_{2}  [x_{1}^5 x_{2}] + \alpha_{3} [x_{1}^4 x_{2}^2 ] + \alpha_{4} [x_{1}^4 x_{2} x_{3}] + \alpha_{5} [x_{1}^{3} x_{2}^{3}] + \alpha_{6} [x_{1}^{3} x_{2}^{2} x_{3}] + \alpha_{7}[x_{1}^{2} x_{2}^{2} x_{3}^{2}].
$$ 
Using the same method as above, we obtain that
$$
\begin{array}{lllllllll}
E([x^{\lambda_1}])
& = & 6[x_1^5] &  = & C_{11} [x^{\mu_1}] ,\\
E([x^{\lambda_2}])
& = & 12[x_1^5]+10 [x_1^4x_2] &  = & C_{12} [x^{\mu_1}] +C_{22}[x^{\mu_2}],\\
E([x^{\lambda_3}])
& = & 10 [x_1^4x_2] +12[x_1^3x_2^2]& = & C_{23} [x^{\mu_2}]+C_{33}[x^{\mu_3}], \\
E([x^{\lambda_4}])
& = & 6 [x_1^4x_2] +12[x_1^3x_2x_3]&  = & C_{24} [x^{\mu_2}]+C_{44}[x^{\mu_4}], \\
E([x^{\lambda_5}])
& = & 12 [x_1^3x_2^2]  & = & C_{35} [x^{\mu_3}], \\
E([x^{\lambda_6}])
& = & 6 [x_1^3x_2^2]  +   20 [x_1^3x_2 x_3]  + 24 [x_1^2x_2^2x_3]  & = & C_{36} [x^{\mu_3}]+  C_{46} [x^{\mu_4}]+  C_{56} [x^{\mu_5}], \\
E([x^{\lambda_7}])
& = & 10 [x_1^2x_2^2x_3]  &  = & C_{57} [x^{\mu_5}].
\end{array} 
$$
Thus
$$
\begin{array}{llllllllll}
C=
\begin{pmatrix}
6 & 12 & 0 & 0 & 0 & 0 & 0 \\
0 & 10 & 10 & 6 & 0 & 0 & 0 \\
0 & 0 & 12 & 0 & 12 & 6 & 0 \\
0 & 0 & 0 & 12 & 0 & 20 & 0 \\
0 & 0 & 0 & 0 & 0 & 24 & 10 
\end{pmatrix},
\end{array}
$$
so we have the following linear system
$$
\alpha_{1} + 2\alpha_{2}=0, \ \ 5 \alpha_{2} + 5 \alpha_{3} + 3 \alpha_{4} = 0,\ \ 2 \alpha_{3} + 2 \alpha_{5} + \alpha_{6} = 0,   \ \ 3\alpha_{4} + 5 \alpha_{6} = 0, \ \ 12\alpha_{6} + 5 \alpha_{7} =0.
$$ 
Since the rank of $C$ is $5$, which is $p_{3,6}(5)$, there are two copies of trivial representation in $\text{Ker}(E) \cap A(7)_{6}$ and we can find explicit basis polynomial of each of them by choosing appropriate values of $a_5$ and $a_7$. (Of course, they should be linearly independent.) One can easily verify that $\text{triv}(7,6) = p_{3, 6} (6) - p_{3, 6}(5) = 7 - 5 = 2$. 
Also we have $C=(C_{pq})$ as follows:
\begin{align*}
&C={\tiny
\begin{pmatrix}
C_{(5,0,0),(6,0,0)}&C_{(5,0,0),(5,1,0)}&0&0&0&0&0\\
0&C_{(4,1,0),(5,1,0)}&C_{(4,1,0),(4,2,0)}&C_{(4,1,0),(4,1,1)}&0&0&0\\
0&0&C_{(3,2,0),(4,2,0)}&0&C_{(3,2,0),(3,3,0)}&C_{(3,2,0),(3,2,1)}&0\\
0&0&0&C_{(3,1,1),(4,1,1)}&0&C_{(3,1,1),(3,2,1)}&0\\
0&0&0&0&0&C_{(2,2,1),(3,2,1)}&C_{(2,2,1),(2,2,2)}
\end{pmatrix}}.
\end{align*}
Then the submatrices $C(a)$ are given as
\begin{align*}
&C(5)=\begin{pmatrix}C_{(5,0,0),(6,0,0)}&C_{(5,0,0),(5,1,0)}&0&0&0&0&0\end{pmatrix},\\
&C(4)=\begin{pmatrix}0&C_{(4,1,0),(5,1,0)}&C_{(4,1,0),(4,2,0)}&C_{(4,1,0),(4,1,1)}&0&0&0\end{pmatrix},\\
&C(3)= \begin{pmatrix}0&0&C_{(3,2,0),(4,2,0)}&0&C_{(3,2,0),(3,3,0)}&C_{(3,2,0),(3,2,1)}&0\\
0&0&0&C_{(3,1,1),(4,1,1)}&0&C_{(3,1,1),(3,2,1)}&0  \end{pmatrix},\\
&C(2)=\begin{pmatrix}0&0&0&0&0&C_{(2,2,1),(3,2,1)}&C_{(2,2,1),(2,2,2)}\end{pmatrix}.
\end{align*}
Note that each $C(a)$ is in row echelon form, and so its rank is equal to 
the number of rows in $C(a)$ for $a=5,4,3,2$. Hence the matrix $C$ is in block lower triangular form and we have 
$$
\text{rank}\, C = \sum_{a=2}^5 \text{rank}\, C(a) = \sum_{a=2}^5 (\text{the number of rows in $C(a)$}) 
= \text{the number of rows in $C$}.
$$

(c) Let $d=7$, $j=8$. Then we have
$$ 
\begin{array}{lllllll}
{\mathscr P}_{3, d-1}(8)
& = & \left\{
\begin{matrix}
\lambda_1=(6,2,0), \lambda_2=(6,1,1), \lambda_3=(5,3,0), \lambda_4=(5,2,1),\\ \lambda_5=(4,4,0),\lambda_6=(4,3,1),\lambda_7=(4,2,2), \lambda_8=(3,3,2)\end{matrix} 
\right\}, \quad \text{and} \\[3ex] 
{\mathscr P}_{3, d-1}(7)
& = & \left\{
\begin{matrix}
\mu_1=(6,1,0), \mu_2=(5,2,0), \mu_3=(5,1,1), \mu_4=(4,3,0), \\
\mu_5=(4,2,1), \mu_6=(3,3,1), \mu_7=(3,2,2)
\end{matrix} 
\right\}.
\end{array} 
$$
Thus the basis polynomial $P$ can be written as 
$$
P= \alpha_1[x_1^6x_2^2] + \alpha_2  [x_1^6x_2x_3] + \alpha_3 [x_1^5 x_2^3 ] + \alpha_4 [x_1^5 x_2^2 x_3] + \alpha_5 [x_1^4 x_2^4] + \alpha_6 [x_1^4 x_2^3 x_3] + \alpha_7[x_1^4 x_2^2 x_3^2]+a_8[x_1^3x_2^3x_3^2].
$$ 
By a simple calculation as above, we get that
$$
\begin{array}{llllllllllllll}
E([x^{\lambda_1}])
& = & 10[x^{\mu_1}]+6[x^{\mu_2}] & = & C_{11}[x^{\mu_1}]+C_{21} [x^{\mu_2}],\\
E([x^{\lambda_2}])
& = & 6[x^{\mu_1}]+6[x^{\mu_3}] & = & C_{12}[x^{\mu_1}]+C_{32} [x^{\mu_3}],\\
E([x^{\lambda_3}])
& = & 12[x^{\mu_2}]+10[x^{\mu_4}] & = & C_{23}[x^{\mu_2}]+C_{43} [x^{\mu_4}],\\
E([x^{\lambda_4}])
& = & 6[x^{\mu_2}]+20[x^{\mu_3}] +10[x^{\mu_5}] & = & C_{24}[x^{\mu_2}]+C_{34} [x^{\mu_3}]+C_{54} [x^{\mu_5}],\\
E([x^{\lambda_5}])
& = & 12[x^{\mu_4}] & = & C_{45}[x^{\mu_4}],\\
E([x^{\lambda_6}])
& = & 6[x^{\mu_4}]+12[x^{\mu_5}] +12[x^{\mu_6}] & = & C_{46}[x^{\mu_4}]+C_{56} [x^{\mu_5}]+C_{66} [x^{\mu_6}],\\
E([x^{\lambda_7}])
& = & 10[x^{\mu_5}]+12[x^{\mu_7}] & = & C_{57}[x^{\mu_5}]+C_{77} [x^{\mu_7}],\\
E([x^{\lambda_8}])
& = & 10[x^{\mu_6}]+12[x^{\mu_7}] & = & C_{68}[x^{\mu_6}]+C_{78} [x^{\mu_8}],
\end{array} 
$$
so
$$
C=
\begin{pmatrix}
10 & 6 & 0 & 0 & 0 & 0 & 0 & 0 \\ 
6 & 0 & 12 & 6 & 0 & 0 & 0 & 0 \\ 
0 & 6 & 0 & 20 & 0 & 0 & 0 & 0 \\ 
0 & 0 & 10 & 0 & 12 & 6 & 0 & 0 \\ 
0 & 0 & 0 & 10 & 0 & 12 & 10 & 0 \\
0 & 0 & 0 & 0 & 0 & 12 & 0 & 10  \\
0 & 0 & 0 & 0 & 0 & 0 & 12 & 12 
\end{pmatrix}.
$$
Since the rank of $C$ is $7$, which is $p_{3,6}(6)$, there is one copy of trivial representation in $\text{Ker}(E) \cap A(7)_{8}$ and we can find explicit basis polynomial by choosing appropriate values of $a_8$. One can easily verify that $\text{triv}(7,8) = p_{3, 6} (8) - p_{3, 6}(7) = 8 - 7 = 1$. 
Also we have $C=(C_{pq})$ as follows:
\begin{align*}
&C=
{\tiny
\begin{pmatrix}
C_{(6,1,0),(6,2,0)}&C_{(6,1,0),(6,1,1)}&0&0&0&0&0&0\\
C_{(5,2,0),(6,2,0)}&0&C_{(5,2,0),(5,3,0)}&C_{(5,2,0),(5,2,1)}&0&0&0&0\\
0&C_{(5,1,1),(6,1,1)}&0&C_{(5,1,1),(5,2,1)}&0&0&0&0\\
0&0&C_{(4,3,0),(5,3,0)}&0&C_{(4,3,0),(4,4,0)}&C_{(4,3,0),(4,3,1)}&0&0\\
0&0&0&C_{(4,2,1),(5,2,1)}&0&C_{(4,2,1),(4,3,1)}&C_{(4,2,1),(4,2,2)}&0\\
0&0&0&0&0&C_{(3,3,1),(4,3,1)}&0&C_{(3,3,1),(3,3,2)}\\
0&0&0&0&0&0&C_{(3,2,2),(4,2,2)}&C_{(3,2,2),(3,3,2)}
\end{pmatrix}}
\end{align*}
Then the submatrices $C(a)$ are given as
\begin{align*}
&C(6)=\begin{pmatrix}C_{(6,1,0),(6,2,0)}&C_{(6,1,0),(6,1,1)}&0&0&0&0&0&0\end{pmatrix},\\
&C(5)=\begin{pmatrix}C_{(5,2,0),(6,2,0)}&0&C_{(5,2,0),(5,3,0)}&C_{(5,2,0),(5,2,1)}&0&0&0&0\\
                                   0&C_{(5,1,1),(6,1,1)}&0&C_{(5,1,1),(5,2,1)}&0&0&0&0\end{pmatrix},\\
&C(4)=\begin{pmatrix}0&0&C_{(4,3,0),(5,3,0)}&0&C_{(4,3,0),(4,4,0)}&C_{(4,3,0),(4,3,1)}&0&0\\
                                   0&0&0&C_{(4,2,1),(5,2,1)}&0&C_{(4,2,1),(4,3,1)}&C_{(4,2,1),(4,2,2)}&0\end{pmatrix},\\
&C(3)= \begin{pmatrix}0&0&0&0&0&C_{(3,3,1),(4,3,1)}&0&C_{(3,3,1),(3,3,2)}\\
                                    0&0&0&0&0&0&C_{(3,2,2),(4,2,2)}&C_{(3,2,2),(3,3,2)}\end{pmatrix}.
\end{align*}
Note that each $C(a)$ is in row echelon form and its rank is equal to 
the number of rows in $C(a)$ for $a=6,5,4,3$. Hence the matrix $C$ is in block lower triangular form and we have 
$$\text{rank}\, C = \sum_{a=3}^6 \text{rank}\, C(a) = \sum_{a=3}^6 (\text{the number of rows in $C(a)$}) 
= \text{the number of rows in $C$}.$$

(d) Let $d=9$, $j=9$. Then we have
$$ 
\begin{array}{lllllll}
{\mathscr P}_{3, d-1}(9)
& = & \left\{
\begin{matrix}
\lambda_1=(8,1,0), \lambda_2=(7,2,0), \lambda_3=(7,1,1), \lambda_4=(6,3,0), \lambda_5=(6,2,1), \lambda_6=(5,4,0),\\ 
\lambda_7=(5,3,1), \lambda_8=(5,2,2), \lambda_9=(4,4,1), \lambda_{10}=(4,3,2),\lambda_{11}=(3,3,3) \end{matrix} 
\right\} \quad \text{and} \\[3ex] 
{\mathscr P}_{3, d-1}(8)
& = & \left\{
\begin{matrix}
\mu_1=(8,0,0), \mu_2=(7,1,0), \mu_3=(6,2,0), \mu_4=(6,1,1), \mu_5=(5,3,0),  \\
\mu_6=(5,2,1), \mu_7=(4,4,0), \mu_8=(4,3,1), \mu_9=(4,2,2), \mu_{10}=(3,3,2)
\end{matrix} 
\right\}.
\end{array} 
$$
Thus the basis polynomial $P$ can be written as 
$$
\begin{array}{lllllllll} 
P
& = & \alpha_1[x_1^8x_2] + \alpha_2  [x_1^7x_2^2] + \alpha_3 [x_1^7 x_2 x_3 ] + \alpha_4 [x_1^6 x_2^3] + \alpha_5 [x_1^6 x_2^2 x_3] + \alpha_6 [x_1^5 x_2^4] + \alpha_7[x_1^5 x_2^3 x_3^2]+\\
&   & \alpha_8[x_1^4x_2^4 x_3 ] +\alpha_9[x_1^4x_2^3 x_3^2 ] +\alpha_{10}[x_1^4x_2^3 x_3^2 ]+\alpha_{11}[x_1^4x_2^3 x_3^2 ].
\end{array} 
$$ 
By a simple calculation as above, we get that
$$
\begin{array}{llllllllllllll}
E([x^{\lambda_1}])
& = & 16[x^{\mu_1}]+8[x^{\mu_2}] & = & C_{11}[x^{\mu_1}]+C_{21} [x^{\mu_2}],\\
E([x^{\lambda_2}])
& = & 14[x^{\mu_2}]+14[x^{\mu_3}] & = & C_{22}[x^{\mu_2}]+C_{32} [x^{\mu_3}],\\
E([x^{\lambda_3}])
& = & 8[x^{\mu_2}]+14[x^{\mu_4}] & = & C_{23}[x^{\mu_2}]+C_{43} [x^{\mu_4}],\\
E([x^{\lambda_4}])
& = & 18[x^{\mu_3}]+18[x^{\mu_5}]  & = & C_{34}[x^{\mu_3}]+C_{54} [x^{\mu_4}],\\
E([x^{\lambda_5}])
& = & 8[x^{\mu_3}] + 28[x^{\mu_4}] + 18[x^{\mu_6}] & = & C_{45}[x^{\mu_4}],\\
E([x^{\lambda_6}])
& = & 20[x^{\mu_5}] +20[x^{\mu_7}] & = & C_{56} [x^{\mu_5}]+C_{76} [x^{\mu_7}],\\
E([x^{\lambda_7}])
& = & 8[x^{\mu_5}]+18[x^{\mu_6}] +20[x^{\mu_8}]& = & C_{57}[x^{\mu_5}]+C_{67} [x^{\mu_6}]+C_{87} [x^{\mu_8}],\\
E([x^{\lambda_8}])
& = & 14[x^{\mu_6}]+20[x^{\mu_9}] & = & C_{68} [x^{\mu_6}]+C_{98}[x^{\mu_9}],\\
E([x^{\lambda_9}])
& = & 8[x^{\mu_7}]+16[x^{\mu_8}] & = & C_{79} [x^{\mu_8}]+C_{89}[x^{\mu_8}]+C_{8,10},\\
E([x^{\lambda_{10}}])
& = & 14[x_{\mu_8}]+36[x^{\mu_9}]+20[x^{\mu_{10}}] & = &C_{8,10} [x^{\mu_8}]+ C_{9,10} [x^{\mu_9}]+C_{10,10}[x^{\mu_{10}}], \\
E([x^{\lambda_{11}}])
& = & 18[x^{\mu_{10}}] & = & C_{10,11}[x^{\mu_{10}}],
\end{array} 
$$
so
$$
C=\left(
\begin{array}{cccccccccccccccccccccccc}
16 &   0 &   0 &   0 &    0 & 0 &   0 &   0 &  0 &   0 & 0 \\ 
8   & 14 &   8 &   0 &   0 &  0 &   0 &   0 &  0 &   0 & 0 \\ 
0   & 14 &   0 & 18 &   8 &  0 &   0 &   0 &  0 &   0 & 0 \\ 
0   &   0 & 14 &  0  & 28 &  0 &   0 &   0 &  0 &   0 & 0 \\ 
0   &   0 &  0 &  18 &  0 & 20 &   8 &   0 &  0 &   0 & 0 \\
0   &   0 & 0 & 0    & 18 &   0 & 18 & 14 &  0 &   0 & 0 \\
0   &   0 & 0 & 0   &   0 &  20 &   0 &   0 &  8 &   0 & 0 \\
0   &   0 & 0 & 0   &   0 &    0 & 20 &   0 &  16 &   14 & 0 \\
0   &   0 & 0 & 0   &   0 &    0 &   0 &   20 & 0  &   36 & 0 \\
0   &   0 & 0 & 0   &   0 &   0 &   0 &  0 &  0 &   20 & 18 
\end{array}\right) .
$$
Since the rank of $C$ is $10$, which is $p_{3,8}(8)$, there is one copy of trivial representation in $\text{Ker}(E) \cap A(9)_{9}$ and we can find explicit basis polynomial by choosing appropriate values of $a_{11}$. One can easily verify that $\text{triv}(9,9) = p_{3, 8} (9) - p_{3, 8}(8) = 11 - 10= 1$. 
Then the submatrices $C(a)$ are given as
\begin{align*}
&C(8)=\begin{pmatrix}C_{(8,0,0),(8,1,0)}&0&0&0&0&0&0&0&0&0~~~0\end{pmatrix},\\
&C(7)=\begin{pmatrix}C_{(7,1,0),(8,1,0)}&C_{(7,1,0),(7,2,0)}&C_{(7,1,0),(7,1,1)}&0&0&0&0&0&0&0~~~0\end{pmatrix},\\
&C(6)=\begin{pmatrix}0&C_{(6,2,0),(7,2,0)}&0&C_{(6,2,0),(6,3,0)}&C_{(6,2,0),(6,2,1)}&0&0&0&0&0~~~0\\ 
                                    0&0&C_{(6,1,1),(7,1,1)}&0&C_{(6,1,1),(6,2,1)}&0&0&0&0&0~~~0\end{pmatrix},\\
&C(5)=\begin{pmatrix}0&0&0&C_{(5,3,0),(6,3,0)}&0&C_{(5,3,0),(5,4,0)}&0&0&0&0~~~0\\
                                   0&0&0&0&C_{(5,2,1),(6,2,1)}&0&C_{(5,2,1),(5,3,1)}&0&0&0~~~0\end{pmatrix},\\
&C(4)=\begin{pmatrix}0&0&0&0&0&C_{(4,4,0),(5,4,0)}&0&0&C_{(4,4,0),(4,4,1)}&0~~~~~~~~~~0\\
                                    0&0&0&0&0&0&C_{(4,3,1),(5,3,1)}&0&C_{(4,3,1),(4,4,1)}&C_{(4,3,1),(4,3,2)}~~~0\\
                                    0&0&0&0&0&0&0&C_{(4,2,2),(5,2,2)}&0&C_{(4,2,2),(4,3,2)}~~~~~~0\end{pmatrix},\\ 
&C(3)=\begin{pmatrix} 0&0&0&0&0&0&0&0&0&C_{(3,3,2),(4,3,2)}~~~C_{(3,3,2),(3,3,3)}\end{pmatrix}.                                                                  
\end{align*}
Note that each $C(a)$ is in row echelon form and its rank is equal to 
the number of rows in $C(a)$ for $a=8,7,6,5,4,3$. Hence the matrix $C$ is in block lower triangular form and we have 
$$\text{rank}\, C = \sum_{a=3}^8 \text{rank}\, C(a) = \sum_{a=3}^8 (\text{the number of rows in $C(a)$}) 
= \text{the number of rows in $C$}.$$
\end{example}

\begin{lem}\label{L:20190930-304} With notations as above, for $d>j>0$, $C$ is a matrix with row echelon form and the rank of $C$ is the number of rows of $C$, which is $p_{3,d-1}(j-1)$. 
\end{lem}

\begin{proof} Let
$$
\begin{array}{rllllllllll}
{\mathscr P}_{3, d-1}(j)
& = & \{\lambda_1=(j,0,0), \lambda_2=(j-1,1,0), \lambda_3=(j-2,2,0), \lambda_4=(j-2,1,1),\dots\}, \\
{\mathscr P}_{3, d-1}(j-1)
& = & \{\mu_1=(j-1,0,0), \mu_2=(j-2,1,0), \mu_3=(j-3,2,0), \mu_4=(j-3,1,1),\dots\}.
\end{array} 
$$
Note that
$$
\begin{array}{llllllllll}
E([x^{\lambda_1}] )
& = & j(d-j) [x^{\mu_1}] & = & C_{11} [x^{\mu_1}], \\
E([x^{\lambda_2}] )
& = & (d-1) [x^{\mu_1}]+(j-1)(d-j+1) [x^{\mu_2}] & = & C_{12} [x^{\mu_1}]+C_{22} [x^{\mu_2}], \\

E([x^{\lambda_3}] )
& = & 2(d-2) [x^{\mu_2}]+(j-2)(d-j+2) [x^{\mu_3}] & = & C_{23} [x^{\mu_2}]+C_{33} [x^{\mu_3}], \\

E([x^{\lambda_4}] )
& = & (d-1) [x^{\mu_2}]+(j-2)(d-j+2) [x^{\mu_4}] & = & C_{24} [x^{\mu_2}]+C_{44} [x^{\mu_4}], \\
&& \hskip 3 true cm \vdots 
\end{array} 
$$
So $C$ is of the form
$$
C=
\begin{pmatrix}
C_{11} & C_{12}   & 0 & 0 & * & \cdots \\
     0    & C_{22} & C_{23} & C_{24} & * & \cdots \\
     0     & 0         &   C_{33} & 0 & * & \cdots \\
     0     & 0         &   0 & C_{44} & * & \cdots \\
     0     & 0         &   0 & 0 & * & \cdots \\
\vdots     & \vdots &   \vdots & \vdots & \vdots & \vdots \\     
     
\end{pmatrix} 
$$

\noindent
{\bfseries\em Claim.} If $C_{i,\ell}\ne 0$ and $C_{k,\ell}=0$ for $k>i$, then $C_{k+1,\ell+1}=0$ for such $k$.

\medskip

\noindent
{\em Proof of Claim.} 
Let $\mu_k=(a,b,c)$ such that $a+b+c=j-1$ and $j-1\ge a\ge b \ge c$. Since $C_{k,j}=0$, we see that no $\lambda_j$ is
$$
(a+1,b,c), (a,b+1,c), (a,b,c+1). 
$$

Note that
$$
\mu_{k+1}=(a,b-1,c+1), (a-1,j-a,0).
$$
If $C_{k+1,j+1}\ne 0$, then $\lambda_{j+1}$ is one of the following, 
$$
(a+1,b-1,c+1), (a,b,c+1), (a,b-1,c+2),
$$
which is impossible. Hence $C_{k+1,j+1}= 0$ for such $k$.

\medskip

It is from Claim that the matrix $C$ is a row echelon form. Moreover, since $E([x^{\lambda_i}])$ generates all partitions in ${\mathscr P}_{3, d-1}(j-1)$, we see that every row of $C$ has non zero component, i.e., the rank of $C$ is the same as the number of rows of $C$, which $p_{3,d-1}(j-1)$, as we wished. 
\end{proof}

\begin{lem}\label{L:20190930-305} With notations as above, for $d\le j$, the rank of $C$ is the number of rows of $C$, which is $p_{3,d-1}(j-1)$. 
\end{lem} 
\begin{proof} 
Let
$$
\begin{array}{rllllllllll}
{\mathscr P}_{3, d-1}(j)
& = & \{\lambda_1=(j-1,1,0), \lambda_2=(j-2,2,0), \lambda_3=(j-2,1,1), \lambda_4=(j-3,3,0),\dots\}, \\
{\mathscr P}_{3, d-1}(j-1)
& = & \{\mu_1=(j-1,0,0), \mu_2=(j-2,1,0), \mu_3=(j-3,2,0), \mu_4=(j-3,1,1),\dots\}.
\end{array} 
$$
Note that
$$
\begin{array}{llllllllll}
E([x^{\lambda_1}] )
& = & 2(d-1) [x^{\mu_1}] +(j-1)(d-j+1) [x^{\mu_2}] & = & C_{11} [x^{\mu_1}]+C_{21} [x^{\mu_2}], \\
E([x^{\lambda_2}] )
& = & 2(d-2) [x^{\mu_2}]+(j-2)(d-j+2) [x^{\mu_3}] & = & C_{22} [x^{\mu_2}]+C_{32} [x^{\mu_3}], \\

E([x^{\lambda_3}] )
& = & (d-1) [x^{\mu_2}]+(j-2)(d-j+2) [x^{\mu_4}] & = & C_{33} [x^{\mu_2}]+C_{43} [x^{\mu_4}], \\

E([x^{\lambda_4}] )
& = & 3(d-3) [x^{\mu_3}]+(j-3)(d-j+3) [x^{\mu_5}] & = & C_{34} [x^{\mu_3}]+C_{54} [x^{\mu_5}], \\
&& \hskip 3 true cm \vdots 
\end{array} 
$$
So $C$ is of the form
$$
C=
\begin{pmatrix}
C_{11} & 0   & 0 & 0 & * & \cdots \\
C_{21}  & C_{22} & 0 & 0 & * & \cdots \\
     0     & C_{32}  &   C_{33} & C_{34} & * & \cdots \\
     0     & 0         &   C_{43} & 0 & * & \cdots \\
     0     & 0         &   0 & C_{54} & * & \cdots \\
\vdots     & \vdots &   \vdots & \vdots & \vdots & \vdots \\     
\end{pmatrix} 
$$

\noindent
{\bfseries\em Claim.} If $C_{i,\ell}\ne 0$ and $C_{k,\ell}=0$ for $k>i\ge 2$, then $C_{k+1,\ell+1}=0$ for such $k$.

\medskip

\noindent
{\em Proof of Claim.} 
Let $\mu_k=(a,b,c)$ such that $a+b+c=j-1$ and $j-1\ge a\ge b \ge c$. Since $C_{k,j}=0$, we see that no $\lambda_j$ is
$$
(a+1,b,c), (a,b+1,c), (a,b,c+1). 
$$

Note that
$$
\mu_{k+1}=(a,b-1,c+1), (a-1,j-a,0).
$$
If $C_{k+1,j+1}\ne 0$, then $\lambda_{j+1}$ is one of the following, 
$$
(a+1,b-1,c+1), (a,b,c+1), (a,b-1,c+2),
$$
which is impossible. Hence $C_{k+1,j+1}= 0$ for such $k$.

\medskip

It is from Claim that the matrix $C$ is a row echelon form. Moreover, since $E([x^{\lambda_i}])$ generates all partitions in ${\mathscr P}_{3, d-1}(j-1)$, we see that every row of $C$ has non zero component, i.e., the rank of $C$ is the same as the number of rows of $C$, which $p_{3,d-1}(j-1)$, as we wished. 
\end{proof}

\begin{theorem} \label{thm:trivial}
{\rm
Given the exponent $d$ and the degree $j$, let 
$$P=\sum_{\lambda = (a, b, c)} \alpha_{a, b, c} [x_{1}^{a}  x_{2}^{b} x_{3}^{c}]$$
be a homogeneous symmetric polynomial of degree $j$, 
where $\lambda = (a,b,c)$ runs over all partitions of $j$ of size $3 \times (d-1)$ and  $[x_{1}^{a}  x_{2}^{b} x_{3}^{c}]$ denotes the $S_{3}$-orbit of $x_{1}^{a} x_{2}^{b} x_{3}^{c}$.
\begin{enumerate}

\item[{\rm (a)}] The solutions of the system of linear equations $E(P)=0$ give explicit basis polynomials for each trivial 
representation in $\text{Ker}(E) \cap A(d)_{j}$.


\item[{\rm (b)}] Let $\text{triv}(d, j)$ denote the multiplicity of trivial representation in the 
irreducible decomposition of $\text{Ker}(E) \cap A(d)_{j}$. Then  
\end{enumerate}
\begin{equation} \label{eq:triv}
\text{triv}(d, j) = p_{3, d-1}(j) - p_{3, d-1}(j-1).
\end{equation}
}
\end{theorem}

\begin{proof} \
The statement (a) is proved in the discussion above the theorem. 

\vskip 2mm 

To prove (b), let $\lambda_{1}, \ldots, \lambda_{t}$ be the partitions in ${\mathscr P}_{3, d-1}(j)$ and let $\mu_{1}, \ldots, \mu_{s}$ be the partitions in ${\mathscr P}_{3, d-1}(j-1)$ listed by the lexicographic ordering. Thus $t=p_{3, d-1}(j)$ and $s=p_{3, d-1}(j-1)$. 

\vskip 2mm

For $\lambda =(a, b, c)$, we denote $x^{\lambda} = x_{1}^{a} x_{2}^{b} x_{3}^{c}$ and write 
$P=\alpha_{1} [x^{\lambda_{1}}] + \cdots + \alpha_{t} [x^{\lambda_{t}}]$ be a homogeneous symmetric polynomial of degree $j$. Then the condition $E(P)=0$ yields a system of linear equations in the variable $\alpha_{1}, \ldots, \alpha_{t}$ with coefficient matrix $C=(C_{pq})$ for $(1\le p \le s,1\le q \le t)$, where $C_{pq}:=C_{\mu_p,\lambda_q}$ such that 
$$
E\left(\alpha_p \left[x^{\lambda_q}\right]\right)=\alpha_p C_{\mu_p,\lambda_q}\left[x^{\mu_p}\right].
$$ 
That is, 
\begin{equation*}
\begin{aligned}
& C_{11} \alpha_{1} + \cdots + C_{1t} \alpha_{t}=0 \\
& C_{21} \alpha_{1} + \cdots + C_{2t} \alpha_{t} =0 \\
& \hskip 15mm  \vdots \\
& C_{s1} \alpha_{1} + \cdots + C_{st} \alpha_{t} =0
\end{aligned}
\end{equation*}

Observe that the coefficients $C_{pq}=0$ unless $\lambda_q=\mu_{p} + (1,0,0)$, $\lambda_q=\mu_{p} + (0,1,0)$, or $\lambda_q=\mu_{p} + (0,0,1)$ and they are partitions. In other words, $C_{pq}=0$ unless the term $[x^{\mu_{p}}]$ arises as an image of $[x^{\mu_{p}+(1,0,0)}]$, $[x^{\mu_{p}+(0, 1,0)}]$, or $[x^{\mu_{p}+(0,0,1)}]$ under $E$. Also note that at least one of them is a partition. (See Lemma \ref{L:20190930-304}, \ref{L:20190930-305}.)

\vskip 2mm 

Let $C(a)$ be the submatrix of $C$ such that i) the non-zero entries are determined by the partitions $\mu \in {\mathscr P}_{3, d-1}(j-1)$, $\lambda \in {\mathscr P}_{3, d-1}(j)$, where the first component of $\mu$ is $a$ and the first component of $\lambda$ is $a$ or $a+1$, ii) the other entries are $0$. (See Example \ref{example3.2} above.)
%
Then each $C(a)$ is in row echelon form. Since $C_{(a,b,c),(a',b',c')}=0$ unless $a'=a$ or $a'=a+1$, the matrix $C$ is in block low triangular form and hence we have
$$\text{rank}\, C = \sum_{a} \text{rank}\, C(a) = \sum_{a} (\text{the number of rows in $C(a)$}) 
= \text{the number of rows in $C$},$$
which proves our claim. 
\end{proof}

\vskip 3mm 

{\bf (2) Sign representation: }

\vskip 2mm 

Given $d$ and $j$, since every permutaion acts as the multiplication by its sign, the candidate polynomial in $\text{Ker}(E) \cap A(d)_{j}$ can be written as 
$$P=\sum_{\mu = (a, b, c)} \beta_{a, b, c} \langle x_{1}^{a}  x_{2}^{b} x_{3}^{c} \rangle,$$
where $\mu = (a,b,c)$ is a {\it strict} partition of $j$ of size $3 \times (d-1)$ and  $\langle x_{1}^{a}  x_{2}^{b} x_{3}^{c} \rangle$ denotes the $S_{3}$-orbit of the monomial $x_{1}^{a} x_{2}^{b} x_{3}^{c}$ (under the sign action). For instance, when $d \ge 4$, we have 
\begin{equation*}
\begin{aligned}
&  \langle x_{1}^3 \rangle=0, \\
& \langle x_{1}^{2} x_{2} \rangle = x_{1}^{2} x_{2} - x_{1}^{2} x_{3} + x_{2}^{2} x_{3} - x_{1} x_{2}^{2} + x_{1} x_{3}^{2} - x_{2} x_{3}^{2}, \\
& \langle x_{1} x_{2} x_{3} \rangle = 0.
\end{aligned}
\end{equation*}

As is the case of trivial representation, the solutions of the system $E(P)=0$ would give  the explicit basis polynomials of each sign representation in $\text{Ker}(E) \cap A(d)_{j}$ and their multipicity is given by  $p_{3, d-1}^{+}(j) - p_{3, d-1}^{+}(j-1)$, where $p_{k,l}^{+}(n)$ denotes the number of strict partitions of size $k \times l$.  Hence we have: 

\vskip 2mm 

\begin{theorem} \label{thm:sign}
{\rm
Given the exponent $d$ and the degree $j$, let 
$$P=\sum_{\mu = (a, b, c)} \beta_{a, b, c} \langle x_{1}^{a}  x_{2}^{b} x_{3}^{c} \rangle$$
be a homogeneous skew-symmetric polynomial of degree $j$,
where $\mu = (a,b,c)$ runs over all rectangular strict partitions of $j$ of size $3 \times (d-1)$ and $\langle x_{1}^{a}  x_{2}^{b} x_{3}^{c} \rangle$ denotes the $S_{3}$-orbit of $x_{1}^{a} x_{2}^{b} x_{3}^{c}$ under the sign action.

\vskip 2mm 

\begin{enumerate}

\item[{\rm (a)}] The solutions of the system of linear equations $E(P)=0$ give  explicit basis polynomials for
each sign representation in $\text{Ker}(E) \cap A(d)_{j}$.

\vskip 2mm 

\item[{\rm (b)}] Let $\text{sign}(d, j)$ denote the multiplicity of sign representation in the irreducible 
decomposition of $\text{Ker}(E) \cap A(d)_{j}$. Then  
\end{enumerate}
\begin{equation} \label{eq:sign}
\text{sign}(d,j)=p_{3, d-1}^{+}(j) - p_{3, d-1}^{+}(j-1).
\end{equation}
}
\end{theorem}

\begin{proof}
The statement (a) is proved in the discussion above the theorem.
The statement (b) can be proved using a similar argument in Theorem 3.1(b). 
\end{proof}

\vskip 2mm

\begin{example} \hfill

\vskip 2mm 

(a) Let $d=5$, $j=4$. Then $\mu = (3, 1, 0)$ is the only strict partition of $4$ of size $3 \times 4$ and hence the only candidate polynomial is 
$$P = x_{1}^{3} x_{2} - x_{1}^{3} x_{3} + x_{2}^{3} x_{3} - x_{1} x_{2}^{3} + x_{1} x_{3}^{3} -x_{2} x_{3}^{3}.$$
But one can check $E(P) \neq 0$, which implies there is no sign representation in $\text{Ker}(E) \cap A(5)_{4}$.

\vskip 2mm 

(b) Let $d=9$, $j=9$.
Then we have 
$$ 
{\mathscr P}^+_{3, d-1}(9)=\{\lambda_1=(8,1,0), \lambda_2=(7,2,0), \lambda_3=(6,3,0), \lambda_4=(6,2,1), \lambda_5=(5,4,0),\lambda_6=(5,3,1),\lambda_7=(4,3,2)\}
$$
and 
$$
{\mathscr P}^+_{3, d-1}(8)=\{\mu_1=(7,1,0), \mu_2=(6,2,0), \mu_3=(5,3,0), \mu_4=(5,2,1), \mu_5=(4,3,1)\}.
$$
Thus the basis polynomial $P$ can be written as
$$P=a_{1} \langle x_{1}^{8} x_{2} \rangle + a_{2} \langle x_{1}^{7} x_{2}^{2} \rangle + a_{3} \langle x_{1}^{6} x_{2}^{3} \rangle + a_{4} \langle x_{1}^{6} x_{2}^{2} x_{3} \rangle + a_{5} \langle x_{1}^{5} x_{2}^{4} \rangle + a_{6} \langle x_{1}^{5} x_{2}^{3} x_{3} \rangle + a_{7} \langle x_{1}^{4} x_{2}^{3} x_{3}^{2} \rangle.$$ 
The condition $E(P)=0$ yields 
$$4 a_{1} + 7 a_{2}=0, \ \ 7a_{2} + 9 a_{3} + 4 a_{4} = 0, \ \ 9 a_{3} + 10 a_{5} + 4 a_{6} =0, \ \ a_{4} + a_{6}=0, \ \ 10 a_{6} + 7 a_{7} =0,$$
and we obtain 
$$a_{1} = - \frac{5}{2} a_{5} - 2a_{6}, \ \ a_{2} = \frac{10}{7} a_{5}  + \frac{8}{7} a_{6}, \ \ a_{3} = - \frac{10}{9} a_{5} - \frac{4}{9} a_{6}, \ \ a_{4} = - a_{6}, \ \ a_{7} = - \frac{10}{7} a_{6}.$$
Hence there are two copies of sign representation in $\text{Ker}(E) \cap A(9)_{9}$ and we can easily find their explicit bases. Note that $\text{sign}(9,9)=p_{3, 8}^{+}(9) - p_{3, 8}^{+}(8) = 7 - 5 =2$ as expected. 

Also we have $C=(C_{pq})$ as follows:
\begin{align*}
&C=
\begin{pmatrix}
C_{11}&C_{12}&C_{13}&C_{14}&C_{15}&C_{16}&C_{17}\\
C_{21}&C_{22}&C_{23}&C_{24}&C_{25}&C_{26}&C_{27}\\
C_{31}&C_{32}&C_{33}&C_{34}&C_{35}&C_{36}&C_{37}\\
C_{41}&C_{42}&C_{43}&C_{44}&C_{45}&C_{46}&C_{47}\\
C_{51}&C_{52}&C_{53}&C_{54}&C_{55}&C_{56}&C_{57}
\end{pmatrix}\\
&\phantom{C}=
{\tiny
\begin{pmatrix}
C_{(7,1,0),(8,1,0)}&C_{(7,1,0),(7,2,0)}&0&0&0&0&0\\
0&C_{(6,2,0),(7,2,0)}&C_{(6,2,0),(6,3,0)}&C_{(6,2,0),(6,2,1)}&0&0&0\\
0&0&C_{(5,3,0),(6,3,0)}&0&C_{(5,3,0),(5,4,0)}&C_{(5,3,0),(5,3,1)}&0\\
0&0&0&C_{(5,2,1),(6,2,1)}&0&C_{(5,2,1),(5,3,1)}&0\\
0&0&0&0&0&C_{(4,3,1),(5,3,1)}&C_{(4,3,1),(4,3,2)}
\end{pmatrix}}
\end{align*}
Then the submatrices $C(a)$ are given as
\begin{align*}
&C(7)=\begin{pmatrix}C_{(7,1,0),(8,1,0)}&C_{(7,1,0),(7,2,0)}&0&0&0&0&0\end{pmatrix},\\
&C(6)=\begin{pmatrix}0&C_{(6,2,0),(7,2,0)}&C_{(6,2,0),(6,3,0)}&C_{(6,2,0),(6,2,1)}&0&0&0\end{pmatrix},\\
&C(5)=\begin{pmatrix}0&0&C_{(5,3,0),(6,3,0)}&0&C_{(5,3,0),(5,4,0)}&C_{(5,3,0),(5,3,1)}&0\\
0&0&0&C_{(5,2,1),(6,2,1)}&0&C_{(5,2,1),(5,3,1)}&0\end{pmatrix},\\
&C(4)= \begin{pmatrix}0&0&0&0&0&C_{(4,3,1),(5,3,1)}&C_{(4,3,1),(4,3,2)}\end{pmatrix}
\end{align*}
Note that each $C(a)$ is in row echelon form and its rank is equal to 
the number of rows in $C(a)$ for $a=6,5,4,3$. Hence the matrix $C$ is in block lower triangular form and we have 
$$\text{rank}\, C = \sum_{a=4}^7 \text{rank}\, C(a) = \sum_{a=4}^7 (\text{the number of rows in $C(a)$}) 
= \text{the number of rows in $C$}.$$

\end{example}

\vskip 3mm

{\bf (3) Standard representation: } 

\vskip 2mm 

Recall that the standard representation is the Specht module corresponding to the partition $\lambda = (2, 1, 0)$ and that it can be realized as the 2-dimensional hyperplane $H=\{(x_{1}, x_{2}, x_{3}) \in  \k^{3} \mid x_{1} + x_{2} + x_{3} = 0 \}$. We can take $u_{1} =(1, -1, 0)$ and $u_{2} = (0, 1, -1)$ as the basis vectors for $H$ and if we set $u_{3}=(-1, 0, 1)$, we have $u_{1} + u_{2} + u_{3} = (0, 0, 0)$. Note that $(1,2)u_{1} = - u_{1}$,  $u_{2} = -(1,3) u_{1}$ and $u_{3} = -(1,2) u_{2}$, where $(i,j)$ denotes he transposition of $i$ and $j$. This observation is the beginning point of our construction. 

\vskip 2mm 

Given $d$ and $j$, let $P_{1}$ be a homogeneous polynomial of degree $j$ such that $(1, 2) P_{1} = -P_{1}$. Set $P_{2} = -(1,3) P_{1}$ and $P_{3} = -(1,2) P_{2}$. Since the symmetric group action commutes with $E$ and $F$, $E(P_{1})=0$ implies $E(P_{2}) = E(P_{3})=0$.  We naturally  take the system of linear equations arising form $E(P_{1})=0$ and $P_{1} + P_{2} + P_{3} =0$. Then the solutions of this system would give the explicit basis polynomials of each standard representation in $\text{Ker}(E) \cap A(d)_{j}$. Moreover, it is straightforward to see that the multiplicity  of the standard representation in $\text{Ker}(E)  \cap A(d)_{j}$ is determined by those of trivial representation and sign representation (and the Hilbert polynomial). 

\vskip 2mm

To summarize, we have:
 
 \vskip 2mm 
 
 \begin{theorem} \label{thm:standard}
 {\rm
 Given the exponent $d$ and the degree $j$, let $P_{1}$ be a homogeneous polynomial of degree $j$ such that $(1,2)P_{1} = -P_{1}$. Set $P_{2} = -(1,3) P_{1}$ and $P_{3}= - (1, 2) P_{2}$. 
 
 \vskip 2mm
 
 \begin{enumerate}
 
 \item[{\rm (a)}] The solutions of the system of linear equations $E(P_{1})=0$ and $P_{1} + P_{2} + P_{3} =0$ give 
 explicit basis polynomials $P_{1}$ and $P_{2}$ for each standard representation 
 in $\text{Ker}(E) \cap A(d)_{j}$. 
 
 \vskip 2mm 
 
 \item[{\rm (b)}] Let $\text{st}(d, j)$ denote the multiplicity of standard representation in 
 the irreducible decomposition of $\text{Ker}(E) \cap A(d)_{j}$. Then 
 \end{enumerate}
 \begin{equation} \label{eq:standard}
 \text{st}(d, j) = \frac{1}{2} \left(\text{mult}(d,j) - \text{triv}(d, j) - \text{sign}(d, j)  \right),
 \end{equation}
 where $\text{mult}(d,j)$, $\text{triv}(d,j)$ and $\text{sign}(d,j)$ are given in 
 \eqref{eq:mult}, \eqref{eq:triv} and \eqref{eq:sign}, respectively. 
 }
 \end{theorem}

\vskip 2mm 

\begin{example} \hfill

\noindent (a) Let $d=3$, $j=3$. Set 
$$P_{1} = a (x_{1}^{2} - x_{1}^{2}) + b (x_{1}^{2} x_{3} - x_{2}^{2} x_{3}) + c (x_{1} x_{3}^{2} - x_{2} x_{3}^{2}),$$
which satisfies the condition $(1, 2) P_{1} = - P_{1}$. The condition $E(P_{1})=0$ yields
$$a + b=0, \ \ b + c =0. $$ 
Thus we may take 
$$P_{1} = x_{1}^{2} x_{2}  - x_{1} x_{2}^{2} - x_{1}^{2} x_{3} + x_{2}^{2} x_{3} + x_{1} x_{3}^{2} - x_{2} x_{3}^{2},$$
from which we obtain
\begin{equation*}
\begin{aligned}
& P_{2} = x_{2}^{2} x_{3} - x_{2} x_{3}^{2} + x_{1} x_{3}^{2} - x_{1} x_{2}^{2}  - x_{1}^{2} x_{3} + x_{1}^{2} x_{2}, \\
& P_{3} = x_{1} x_{3}^{2} - x_{1}^{2} x_{3} - x_{2} x_{3}^{2} + x_{1}^{2} x_{2} + x_{2}^{2} x_{3} - x_{1} x_{2}^{2}.
\end{aligned}
\end{equation*}
But $P_{1} + P_{2} + P_{3} \neq 0$ and hence there is no standard representation in $\text{Ker}(E) \cap A(3)_{3}$. 

\vskip 2mm 

\noindent (b) Let $d=5$, $j=4$. Using the condition $(1,2) P_{1} = - P_{1}$, we may write 
\begin{equation*}
\begin{aligned}
P_{1} = & a (x_{1}^{4} - x_{2}^{4}) + b (x_{1}^{3} x_{2} - x_{1} x_{2}^{3}) + c (x_{1}^{3} x_{3} - x_{2}^{3} x_{3}) + d (x_{1}^{2} x_{3}^{2} - x_{2}^{2} x_{3}^{2}) \\
& + e (x_{1}x_{3}^{3} - x_{2} x_{3}^{3}) + f (x_{1}^{2} x_{2} x_{3} - x_{1} x_{2}^{2} x_{3}).
\end{aligned}
\end{equation*}
Then $E(P_{1})=0$ and $P_{1} + P_{2} + P_{3} =0$ yield
$$ a + b + c=0, \ \ 3b + 2f =0, \ \ c + 3d + 2f =0, \ \ d+e =0.$$
Hence we obtain 
$$a = -e -  \frac{4}{3} f, \ \ b = - \frac{2}{3} f, \ \ c = e - \frac{2}{3} f, \ \ d = -e.$$

Therefore $\text{st}(5, 4)=2$ and we can compute two pairs of polynomials that provide explicit basis polynomials for  each  standard representation in $\text{Ker}(E) \cap A(5)_{4}$: 

(i) if we take $e= -1$, $f=0$, then we have 
\begin{equation*}
\begin{aligned}
& P_{1} = (x_{1}^{4} - x_{2}^{4}) - (x_{1}^{3} x_{3}  - x_{2}^{3} x_{3}) + (x_{1}^{2} x_{3}^{2} - x_{2}^{2} x_{3}^{2}) -( x_{1} x_{3}^{3} - x_{2} x_{3}^{3}), \\
&  P_{2} = (x_{2}^{4} - x_{3}^{4}) -( x_{1} x_{2}^{3}  - x_{1} x_{3}^{3})  +(x_{1}^{2} x_{2}^{2} - x_{1}^{2} x_{3}^{2}) - (x_{1}^{3} x_{2}  - x_{1}^{3} x_{3}).
\end{aligned}
\end{equation*}

(ii) if we take $e= 0$, $f=-3$, then we have 
\begin{equation*}
\begin{aligned}
& P_{1} =4(x_{1}^{4} - x_{2}^{4}) + 2( x_{1}^{3} x_{2} -x_{1} x_{2}^{3} ) + 2( x_{1}^{3} x_{3} - x_{2}^{3} x_{3}) - 3(x_{1}^{2} x_{2} x_{3} - x_{1} x_{2}^{2} x_{3}), \\
& P_{2} =4(x_{2}^{4} - x_{3}^{4}) + 2( x_{2}^{3} x_{3} -x_{2} x_{3}^{3} ) + 2( x_{1} x_{2}^{3} - x_{1} x_{3}^{3}) - 3(x_{1} x_{2}^{2}  x_{3} - x_{1} x_{2} x_{3}^{2}).
\end{aligned}
\end{equation*}

\end{example}

\vskip 5mm

\section{ Multiplicities of irreducible submodules: Recursive formulas}

In this section, we will give recursive formulas for the multiplicities of the irreducible submodules in 
$\text{Ker}(E) \cap A(d)_{j}$ for all $d \ge 3$, $j \ge 0$. 

\vskip 2mm

Recall that we denote by ${\mathscr P}_{k,l}(j)$ (resp. ${\mathscr P}_{k,l}^{+}(j)$) be the set of rectangular partitions 
(resp. strict partitions) of $j$ of size $k \times l$, and let $p_{k,l}(j)$ (resp. $p_{k,l}^{+}(j)$) be the number of 
elements in   ${\mathscr P}_{k,l}(j)$ (resp. ${\mathscr P}_{k,l}^{+}(j)$).  We say that a partition $\lambda = (\lambda_{1} 
\ge \lambda_{2} \ge \ldots \ge \lambda_{r} \ge 0)$ is {\it $2$-strict} if $\lambda_{j} - \lambda_{j+1} \ge 2$ for all
$1 \le j \le r$. We will use the notation ${\mathscr P}^{+\!+}_{k,l}(j)$ for the set of $2$-strict rectangular partitions 
of size $k \times l$ and $p_{k,l}^{+\!+}(j)$ the number of elements in   ${\mathscr P}^{+\!+}_{k,l}(j)$. 

\vskip 2mm 

Let $\Delta = (k-1, k-2, \ldots, 1, 0)$ be a rectangular partition of $\frac{k(k-1)}{2}$ of size $k \times (k-1)$. 
Then it is straightforward to see that there is a bijection 
\begin{equation} \label{eq:strict}
{\mathscr P}^{+}(j) \overset {\sim} \longrightarrow{\mathscr P}_{k, l-k+1} \left(j-\frac{k(k-1)}{2}\right) \quad \text{given by}
\ \ \lambda \longmapsto \lambda - \Delta.
\end{equation}

Similarly, we have
\begin{equation} \label{eq:2-strict}
{\mathscr P}_{k,l}^{+\!+} (j) \overset{\sim} \longrightarrow {\mathscr P}_{k, l-k+1}^{+} \left( j - \frac{k(k-1)}{2} \right).
\end{equation} 

In particular, when $k=3$, $l=d-1$, we have 
\begin{equation} \label{eq:n=3}
{\mathscr P}_{3, d-1}^{+}(j) \overset {\sim} \longrightarrow {\mathscr P}_{3, d-3}(j-3), \quad 
{\mathscr P}_{3, d-1}^{+\!+}(j) \overset {\sim} \longrightarrow {\mathscr P}_{3, d-3}^{+}(j-3).
\end{equation}

Therefore, we obtain the following recursive relations.

\vskip 2mm

\begin{proposition} \label{prop:sign-triv}
{\rm For all $d \ge 3$, $j \ge 0$, we have 
\begin{equation} \label{eq:sign-triv}
\text{sign}(d,j) = \text{triv}(d-2, j-3).
\end{equation}
}
\end{proposition}

{\it Proof.} \ By Theorem \ref{thm:sign} and \eqref{eq:strict}, we have 
\begin{equation*}
\text{sign}(d,j) = p_{3, d-1}^{+}(j) - p_{3, d-1}^{+}(j-1) = p_{3, d-3}(j-3) - p_{3, d-3}(j-4) 
=\text{triv}(d-2, j-3).  \qed
\end{equation*}

\vskip 2mm 

In the next proposition, we will prove the opposite direction of recursive relations.

\vskip2mm

\begin{proposition} \label{prop:triv-sign}
{\rm For all $d \ge 3$, $j \ge 0$, we have 
\begin{equation} \label{eq:triv-sign}
\text{triv}(d,j) = \begin{cases} 
\, \text{sign}(d-2, j-3) + 1 \ \ & \text{if} \ j \le d-1, \\
\, \text{sign}(d-2, j-3) \ \ & \text{if} \ j \ge d. 
\end{cases}
\end{equation}
}
\end{proposition} 

{\it Proof.} \ Note that 
\begin{equation*}
p_{3,d-1}^{+\!+}(j) - p_{3, d-1}^{+\!+}(j-1) = p_{3,d-3}^{+}(j-3) - p_{3,d-3}^{+}(j-4) = \text{sign}(d-2, j-3).
\end{equation*}

Hence our claim is equivalent to the relations
\begin{equation} \label{eq:triv-sign2}
p_{3, d-1}(j) - p_{3, d-1} (j-1) 
= \begin{cases}
p_{3,d-1}^{+\!+}(j) - p_{3, d-1}^{+\!+}(j-1) +1 \ \ & \text{if} \ j \le d-1, \\
p_{3, d-1}^{+\!+}(j) - p_{3,d-1}^{+\!+}(j-1) \ \ & \text{if} \ j \ge d.
\end{cases} 
\end{equation}

We will prove our claim in 4 steps. 

\vskip 2mm 

{\bf Step 1:} \ For all $d \ge 3$, $j \ge 0$, we have 
\begin{align*}
& (p_{3,d-1}(j) - p_{3,d-1}(j-1)) - (p_{3, d-1}^{+\!+}(j) - p_{3,d-1}^{+\!+}(j-1)) \\
& = \# \{\lambda = (\lambda_{1}, \lambda_{2}, \lambda_{3} \in {\mathscr P}_{3, d-1}(j) 
\mid \lambda_{1} = \lambda_{2} \} \\
&\phantom{=} - \# \{ \lambda = (\lambda_{1}, \lambda_{2}, \lambda_{3} \in {\mathscr P}_{3, d-1}(j) \mid 
\lambda_{1} = \lambda_{2} + 2, \, \text{$\lambda$ is $2$-struct}  \} \\
&\phantom{=} - \# \{\mu = (\mu_{1}, \mu_{2}, \mu_{3}) \in {\mathscr P}_{3,d-1}(j-1) \mid
\mu_{1} = d-1, \, \text{$\mu$ is not $2$-strict}  \}.
\end{align*}

\vskip 2mm 

Roughly speaking, our main idea is to use the following map 
 \begin{equation*}
 {\mathscr P}_{3,d-1}(j) \ni \lambda = (\lambda_{1}, \lambda_{2}, \lambda_{3}) 
 \longmapsto \mu=(\lambda_{1} -1, \lambda_{2}, \lambda_{3}) \in  {\mathscr P}_{3,d-1}(j-1) 
 \end{equation*}
whenever it makes sense and add the missing partitions.  For simplicity, we will just use $\lambda$
and $\mu$ for the partitions in ${\mathscr P}_{3, d-1}(j)$ and ${\mathscr P}_{3, d-1}(j-1)$, respectively. 

\vskip 2mm 

We first observe 
\begin{align*}
& p_{3,d-1}(j) - p_{3, d-1}(j-1) \\
& = \# \{\ \lambda \mid \lambda_{1} = \lambda_{2}\}
+ \# \{ \lambda \mid \lambda_{1} = \lambda_{2}+1 \} 
 - \# \{\mu \mid \mu_{1} = \mu_{2} \}  \\
& \ \  + \# \{\lambda \mid \lambda_{1} = \lambda_{2} + 2 \} -  \# \{\mu \mid \mu_{1} = \mu_{2}  + 1 \} \\
& \ \ + \# \{\lambda \mid \lambda_{1} \ge  \lambda_{2} + 3, \, \text{$\lambda$ is 2-strict} \}
+ \#(\{\lambda \mid \lambda_{1} \ge  \lambda_{2} + 3, \, \text{$\lambda$ is not 2-strict} \} \\
& \ \ -  \# \{\mu \mid \text{$\mu$ is 2-strict} \} -  \#(\{\mu \mid \mu_{1} \ge \mu_{2} + 2, \, \text{$\mu$ is not 2-strict} \} \\
& = \# \{\ \lambda \mid \lambda_{1} = \lambda_{2}\} - \# \{\mu \mid \mu_{1} = d-1 \}.
\end{align*}

On the other hand, we have 
\begin{align*}
& p_{3,d-1}^{+\!+}(j) - p_{3, d-1}^{+\!+}(j-1) \\
&= \# \{ \lambda \mid  \lambda_{1} \ge \lambda_{2} + 3, \, \text{$\lambda$ is 2-strict} \} 
 - \# \{ \mu \mid \text{$\mu$ is 2-strict} \} + \# \{\lambda \mid \lambda_{1} = \lambda_{2} + 2, \, 
\text{$\lambda$ is 2-strict} \} \\
& = \# \{ \lambda \mid \lambda_{1} = \lambda_{2} + 2, \, \text{$\lambda$ is 2-strict} \} 
- \# \{ \mu \mid \mu_{1} = d-1, \, \text{$\mu$ is 2-strict} \}. 
\end{align*}

Hence we obtain 
\begin{align*}
&  (p_{3,d-1}(j) - p_{3,d-1}(j-1)) - (p_{3, d-1}^{+\!+}(j) - p_{3,d-1}^{+\!+}(j-1)) \\
& = \# \{\lambda \mid \lambda_{1} = \lambda_{2} \} 
 - \# \{ \lambda  \mid \lambda_{1} = \lambda_{2} + 2, \, \text{$\lambda$ is $2$-struct}  \}   - \# \{\mu \mid \mu_{1} = d-1, \, \text{$\mu$ is not $2$-strict}  \}.
\end{align*}

\vskip 3mm 

{\bf Step 2:} If $\lambda = (a,a,b) \in {\mathscr P}_{3,d-1}(j)$ for $d \ge 3$, $j \ge 0$, then $a+1 \le d-1$. 

\vskip 2mm 

To prove this, since $j=2a+b$. we have only to check our claim when $j$ and $d$ are maximal.

\vskip 2mm 

If $d=3$, then $j=3$ $a=1$, which verifies our claim. 
If $d$ is odd $\ge 5$, then we have 
$$j = \dfrac{3(d-1)}{2} \ \ \text{and} \ \ a \le \dfrac{3(d-1)}{4}.$$
Hence  we obtain 
$$(d-1) - (a+1) \ge (d-1) - \left(\frac{3(d-1)}{4} +1 \right) = \frac{d-5}{4} \ge 0.$$

If $d$ is even, then $d \ge 4$ and 
$$j = \left\lfloor \frac{3(d-1)}{2} \right\rfloor = \frac{3d-4}{2}, \quad a \le \frac{3d-4}{4}.$$
Hence we obtain 
$$(d-1) - (a+1) \ge (d-1) - \left(\frac{3d-4}{4} + 1 \right) = \frac{d-4}{4} \ge 0,$$
which proves our claim. 

\vskip 3mm 

{\bf Step 3:} $\# \{\lambda \in {\mathscr P}_{3,d-1}(j) \mid \lambda_{1} = \lambda_{2} \} 
= \# \{\lambda \in {\mathscr P}_{3,d-1}(j) \mid \lambda_{1} = \lambda_{2} + 2, \, 
\text{$\lambda$ is 2-strict} \} + 1.$

\vskip 3mm 

Set 
\begin{equation*}
\begin{aligned}
& L=\{\lambda  \in {\mathscr P}_{3,d-1}(j) \mid \lambda_{1} = \lambda_{2} \}, \ \ \ 
L^{+} = \{\lambda  \in {\mathscr P}_{3,d-1}(j) \mid \lambda = \lambda_{2} \ge \lambda_{3} +3 \},  \\
& R = \{ \lambda  \in {\mathscr P}_{3,d-1}(j) \mid \lambda_{1}=\lambda_{2} + 2, \, \text{$\lambda$ is 2-strict} \}.
\end{aligned}
\end{equation*}

\vskip 2mm 

Then it is easy to see that there is a bijection 
$$L^{+} \overset {\sim} \longrightarrow R \quad \text{given by} \ \ 
\lambda = (a,a,b) \longmapsto \lambda' = (a+1, a-1, b),$$
which implies  $L \setminus L^{+}  = \{ \lambda = (a,a,b) \mid 
\lambda_{1} = \lambda_{2} \le \lambda_{3} +2 \}.$

\vskip 2mm 

Therefore, we have the following three cases:

\vskip 2mm

\hskip 3mm (i) if $j \equiv 1 \  \ (\text{mod} \ 3)$, then $a=b+2$ and $L \setminus L^{+} = \{(a,a,a-2)\}$,  

\vskip 2mm 

\hskip 3mm (ii) if $j \equiv 2 \  \ (\text{mod} \ 3)$, then $a=b+1$ and $L \setminus L^{+} = \{(a,a,a-1)\}$,  

\vskip 2mm 

\hskip 3mm (iii) if $j \equiv 0 \  \ (\text{mod} \ 3)$, then $a=b$ and $L \setminus L^{+} = \{(a,a,a)\}$.

\vskip 2mm 

\noindent 
Hence $\# (L \setminus L^{+}) =1$.

\vskip 3mm 

{\bf Step 4:} 

\vskip 2mm 

(a) If $j \le d-1$, then $\{ \mu \in {\mathscr P}_{3,d-1}(j-1) \mid \mu_{1} = d-1, \, 
\text{$\mu$ is not 2-strict} \} = \emptyset$. 

\vskip 2mm 

(b) If $j \ge d$, then $\# (\{ \mu \in {\mathscr P}_{3,d-1}(j-1) \mid \mu_{1} = d-1, \, 
\text{$\mu$ is not 2-strict} \}) = 1$. 

\vskip 2mm 

The statement (a) is clear because $\mu = (d-1, a,b)$ cannot be a partition of $j-1$. 

\vskip 2mm 

Suppose $j \ge d$. We begin with $\mu_{0} = (d-1, j-d, 0)$. If $\mu_{0}$ is not 2-strict,
then $j-d=1$ and we stop there. Otherwise, we continue to obtain 
$\mu_{k}=(d-1, j-d-k, k)$, where  $k$ is the largest non-negative integer such that 
$\mu_{k}$ is 2-strict. Then $j-d-k \ge k+2$ and $\mu_{k+1}=(d-1, j-d-k-1, k+1)$,
which is not 2-strict; i.e., $j-d-k-1 \le 1$. Then $\mu_{k+2} = (d-1, j-d-k-2, k+2)$ is not
even a partition and our assertion is proved. 

\vskip 3mm 

Combining (Step 1) - (Step 4), we obtain   \eqref{eq:triv-sign2}. \qed

\vskip 3mm 

\begin{corollary} \label{cor:st-recursive}
{\rm
For all $d\ge 3$, $j \ge 0$, we have 
\begin{equation}  \label{eq:st-recursive}
\text{st}(d,j) = 
\begin{cases}
\text{st}(d-2, j-3) + 1  \ \ & \text{if} \ j \le d-1, \\
\text{st}(d-2, j-3) \ \ & \text{if}  \ j \ge d.
\end{cases}
\end{equation}
}
\end{corollary}

\begin{proof} \ Note that 
\begin{equation} \label{eq:st-mult}
\begin{aligned}
& 2 (\text{st}(d,j) - \text{st}(d-2, j-3)) 
+ (\text{triv}(d,j) - \text{triv}(d-2, j-3)) 
+ (\text{sign}(d,j) - \text{sign}(d-2,j-3)) \\
&
 = \text{mult}(d,j) - \text{mult}(d-2,j-3).
\end{aligned}
\end{equation}

Let $L$ (resp. $R$) denote the left-hand side (resp. right-hand side) of \eqref{eq:st-mult}. 

\vskip 2mm 

If $j \le d-1$, then 
\begin{equation*}
L = 2 ( \text{st}(d,j) - \text{st}(d-2, j-3) ) + 1, \ \ R = (j+1) - (j-2) = 3.
\end{equation*}
Hence we get 
$$\text{st}(d,j) - \text{st}(d-2,j-3) =1.$$

If $j \ge d+1$, then 
$$L = 2 ( \text{st}(d,j) - \text{st}(d-2, j-3) ), \ \ R=(3d-2-2j) - (3(d-2)-2-2(j-3))=0,$$
which yields 
$$\text{st}(d,j) - \text{st}(d-2,j-3) =0.$$

If $j=d$, then 
$$L = 2 ( \text{st}(d,d) - \text{st}(d-2, d-3) ), \ \ R=(3d-2-2d)-(d-2)=0,$$
which completes the proof. 
\end{proof}

\vskip 2mm 

\begin{example} \ In this example, we illustrate how to compute multiplicities  
irreducible components using the resursive formulas \eqref{eq:sign-triv}, \eqref{eq:triv-sign} and 
\eqref{eq:st-recursive}.

\vskip 2mm 

(a) Suppose that $d=9$. 

\vskip 2mm 

\hskip 4mm (i) If $j=8$, then we have
\begin{equation*}
\begin{aligned}
& \text{triv}(9,8) = \text{sign}(7,5) +1 = \text{triv} (5, 2) +1 = \text{sign}(3, -2) +2 =2, \\
& \text{sign}(9,8) = \text{triv}(7,5) =\text{sign} (5,2) +1 = \text{triv}(3,-2) +1 =1, \\
& \text{st}(9,8) = \text{st}(7,5) + 1 = \text{st}(5,2)  + 2 =1+2= 3. 
\end{aligned}
\end{equation*}

\hskip 4mm (ii) If $j=12$, then we have
\begin{equation*}
\begin{aligned}
& \text{triv}(9,12) = \text{sign}(7,9)  = \text{triv} (5, 6)  = \text{sign}(3, 3) =1, \\
& \text{sign}(9,12) = \text{triv}(7,9) =\text{sign} (5,6) = \text{triv}(3,3) =0, \\
& \text{st}(9,12) = \text{st}(7,9)  = \text{st}(5,6)  = \text{t}(3,3)  = 0. 
\end{aligned}
\end{equation*}

\hskip 4mm (iii) If $j=6$, then we have
\begin{equation*}
\begin{aligned}
& \text{triv}(9,6) = \text{sign}(7,3) +1 = \text{triv} (5, 0) +1 = 1=1 =2, \\
& \text{sign}(9,6) = \text{triv}(7,3) =\text{sign} (5,0) +1 = 1, \\
& \text{st}(9,6) = \text{st}(7,3) + 1 = \text{st}(5,0)  + 2 = 2. 
\end{aligned}
\end{equation*}

\vskip 3mm 

(b) Suppose that $d=10$. 

\vskip 2mm 

\hskip 4mm (i) If $j=9$, then we have
\begin{equation*}
\begin{aligned}
& \text{triv}(10,9) = \text{sign}(8,6) +1 = \text{triv} (6, 3) +1 = \text{sign}(4, 0) +2 =2, \\
& \text{sign}(10,9) = \text{triv}(8,6) =\text{sign} (6,3) +1 = \text{triv}(4,0) +1 =1+1=2, \\
& \text{st}(10,9) = \text{st}(8,6) + 1 = \text{st}(6,3)  + 2 = \text{st}(4,0) +3 = 3. 
\end{aligned}
\end{equation*}

\hskip 4mm (ii) If $j=12$, then we have
\begin{equation*}
\begin{aligned}
& \text{triv}(10,12) = \text{sign}(8,9)  = \text{triv} (6, 6)  = \text{sign}(4, 3) =1, \\
& \text{sign}(10,12) = \text{triv}(8,9) =\text{sign} (6,6) = \text{triv}(4,3) =1, \\
& \text{st}(10,12) = \text{st}(8,9)  = \text{st}(6,6)  = \text{st}(4,3)  = 1. 
\end{aligned}
\end{equation*}

\hskip 4mm (iii) If $j=6$, then we have
\begin{equation*}
\begin{aligned}
& \text{triv}(10,6) = \text{sign}(8,3) +1 = \text{triv} (5, 0) +1 = 1+1 =2, \\
& \text{sign}(10,6) = \text{triv}(8,3) =\text{sign} (6,0) +1 = 1, \\
& \text{st}(10,6) = \text{st}(8,3) + 1 = \text{st}(6,0)  + 2 = 2. 
\end{aligned}
\end{equation*}

\end{example}

\vskip 5mm 

\section{Multiplicities of irreducible submodules: closed form formulas}

In this section, using the recursive formulas in Proposition \ref{prop:sign-triv}, Proposition \ref{prop:triv-sign}
and Corollary \ref{cor:st-recursive}, we derive explicit closed form formulas for the multiplicities of
irreducible submodules in $\text{Ker}(E) \cap A(d)_{j}$ for all $d \ge 3$, $j \ge 0$. 

\vskip 2mm 

We first take care of the case when $j=d-1$, which will play the role of corner-stone of our closed form formulas.

\vskip 2mm 

\begin{lemma} \label{lem:j=d-1} \hfill

\vskip 2mm 

{\rm 
(a) If $d$ is odd $\ge 3$, then we have 
$$\text{triv}(d, d-1) = \left\lfloor \frac{d-1}{6}\right\rfloor +1, \ \ \text{sign}(d, d-1) = \left\lfloor \frac{d-1}{6} \right\rfloor .$$

(b) If $d$ is even $\ge 4$, then we have 
$$\text{triv}(d, d-1) = \text{sign}(d,d-1) = \left\lfloor \frac{d+2}{6}\right\rfloor.$$

}
\end{lemma}

\begin{proof}
Since $d-1  \le d$, by the recursive formulas \eqref{eq:sign-triv} and \eqref{eq:triv-sign},
we have 
\begin{equation*}
\begin{aligned}
& \text{triv}(d, d-1) = \text{sign}(d-2, d-4) + 1 \\
&\ \ = \text{triv}(d-4, d-7) + 1 = \text{sign}(d-6,d-10) + 2 \\
& \ \ = \cdots = \text{sign}(d-4l+2, d-6l+2) + l = \text{triv}(d-4l, d-6l -1) + l, \\
& \text{sign}(d, d-1) = \text{triv}(d-2, d-4)  \\
&\ \  = \text{sign}(d-4, d-7) + 1 = \text{triv}(d-6,d-10) + 1 \\
& \ \ = \cdots = \text{triv}(d-4l+2, d-6l+2) + l -1 = \text{sign}(d-4l, d-6l-1) + l.  
\end{aligned}
\end{equation*}

\vskip 2mm 

(a) Suppose $d$ is odd and $d \ge 3$. 

\vskip 2mm 

If $d = 6l +1$, then 
\begin{equation*}
\begin{aligned}
& \text{triv}(d,d-1) = \text{triv}(2l+1, 0) + l =l+1 =\frac{d-1}{6} + 1 = \left\lfloor \frac{d-1}{6} \right\rfloor +1, \\
& \text{sign}(d, d-1) = \text{sign}(2l+1,0) + l =l = \left\lfloor \frac{d-1}{6} \right\rfloor.
\end{aligned}
\end{equation*}

If $d=6l+3$, then 
\begin{equation*}
\begin{aligned}
& \text{triv}(d,d-1) = \text{triv}(2l+3, 2) + l =l+1 =\frac{d-3}{6} + 1  = \left\lfloor \frac{d-3}{6} \right\rfloor +1\\
& \ \  = \left\lfloor \frac{d-3}{6} + \frac{1}{3}\right\rfloor +1 = \left\lfloor \frac{d-1}{6} \right\rfloor +1,\\
& \text{sign}(d, d-1) = \text{sign}(2l+3,2) + l = l = \left\lfloor \frac{d-3}{6} \right\rfloor = \left\lfloor \frac{d-1}{6} \right\rfloor.
\end{aligned}
\end{equation*}

If $d=6l+5$, then 
\begin{equation*}
\begin{aligned}
& \text{triv}(d,d-1) = \text{triv}(2l+5, 4) + l = \text{sign} (2l+3, 1) + l + 1\\
& \ \ = l+1  = \left\lfloor \frac{d-5}{6}\right\rfloor +1 =   \left \lfloor \frac{d-1}{6} \right\rfloor +1, \\
& \text{sign}(d, d-1) = \text{sign}(2l+5,4) + l 
= \text{triv}(2l+3, 1) + l = l = \left \lfloor \frac{d-1}{6} \right \rfloor.
\end{aligned}
\end{equation*}

\vskip 2mm 

(b) Suppose $d$ is even and $d \ge 4$. 

\vskip 2mm

If $d=6l$, then 
\begin{equation*}
\begin{aligned}
& \text{triv}(d,d-1) = \text{triv}(2l, -1) + l = l =\frac{d}{6} = \left\lfloor \frac{d}{6} \right \rfloor = \left \lfloor \frac{d+2}{6} \right \rfloor, \\ 
& \text{sign}(d, d-1) = \text{sign}(2l, -1) +l = l = \left\lfloor \frac{d+2}{6} \right \rfloor.
\end{aligned}
\end{equation*}

If $d = 6l + 2$, then 
\begin{equation*}
\begin{aligned}
& \text{triv}(d,d-1) = \text{triv}(2l+2, 1) + l = l =\frac{d-2}{6} = \left\lfloor \frac{d-2}{6} \right \rfloor =\left \lfloor \frac{d+2}{6} \right\rfloor, \\ 
& \text{sign}(d, d-1) = \text{sign}(2l+2, 1) +l = l = \left \lfloor \frac{d+2}{6} \right \rfloor.
\end{aligned}
\end{equation*}

If $d = 6l + 4$, then 
\begin{equation*}
\begin{aligned}
& \text{triv}(d,d-1) = \text{triv}(2l+4, 3) + l = \text{sign}(2l+2, 0) +l =\frac{d-2}{6} = \left \lfloor \frac{d-2}{6} \right \rfloor
 = \left \lfloor \frac{d+2}{6} \right \rfloor, \\ 
& \text{sign}(d, d-1) = \text{sign}(2l + 4, 3) +l = l +1 =\left \lfloor \frac{d-4}{6} +1  \right \rfloor 
= \left \lfloor \frac{d+2}{6} \right \rfloor.
\end{aligned}
\end{equation*}
\end{proof}

\vskip 2mm 

\begin{example} \ Lemma \ref{lem:j=d-1} implies
\begin{equation*}
\begin{aligned}
& \text{triv} (9,8) = \left\lfloor \frac{9-1}{6} \right\rfloor + 1 =2, \\
& \text{sign}(9,8) = \left\lfloor \frac{9-1}{6} \right \rfloor =1, \\
& \text{triv}(10,9) = \text{sign} (10,9) = \left\lfloor \frac{10+2}{6} \right \rfloor =2. 
\end{aligned}
\end{equation*}

\end{example}

\vskip 3mm 

Now using Lemma \ref{lem:j=d-1}, we will derive the general multiplicity formulas 
for $\text{triv}(d,j)$ and $\text{sign}(d,j)$. 
Note that, thanks to recursive formulas \eqref{eq:sign-triv} and \eqref{eq:triv-sign},  
we have the interlocking relations
\begin{equation}\label{eq:interlocking}
\begin{aligned}
& \text{triv}(d,j) = \text{sign}(d+2, j+3) \ \ \text{for all}  \ d\ge 3, j\ge 0, \\
& \text{sign}(d,j) = \begin{cases}
\text{triv}(d+2, j+3) -1 \ & \text{if} \ j \le d-2, \\
\text{triv}(d+2, j+3) \ & \text{if} \ j \ge d-1.
\end{cases} 
\end{aligned}
\end{equation}

\vskip 2mm 

Suppose $j \ge d$. Then we have 
\begin{equation*}
\begin{aligned}
& \text{triv}(d, j) = \text{sign}(d-2, j-3)  \\
&\ \ = \text{triv}(d-4, j-6)  = \text{sign}(d-6,j-9)  \\
& \ \ = \cdots = \text{sign}(d-4l+2, j-6l+3) = \text{triv}(d-4l, j-6l) = \cdots, \\
& \text{sign}(d, j) = \text{triv}(d-2, j-3)  \\
&\ \  = \text{sign}(d-4, j-6) = \text{triv}(d-6,j-9)  \\
& \ \ = \cdots = \text{triv}(d-4l+2, d-6l+3)  = \text{sign}(d-4l, j-6l) = \cdots.  
\end{aligned}
\end{equation*}

We continue this process until we reach the point where 
$$d - 4l + 1 = j-6l + 3 \ \ \text{or} \ \ d - 4l -1 = j - 6l.$$

If $d - 4l + 1 = j - 6l +3$, then $j-d = 2(l-1)$ is even and $l = \dfrac{j-d+2}{2}$. 

If $d-4l-1 = j-6l$, then $j-d = 2l-1$ is odd and $l = \dfrac{j-d+1}{2}$. 

\vskip 3mm 

Now we will compute $\text{triv}(d,j)$ and $\text{sign}(d,j)$ when $j \ge d$. 

\vskip 3mm 

{\it Case 1}: \, If $d$ is odd and $j-d$ is even, then we have 
\begin{equation*}
\begin{aligned}
& \text{triv}(d,j) = \text{sign}(d-4l+2, j - 6l+3) 
=\left\lfloor  \frac{d-4l+1}{6} \right\rfloor \\
& \ \ =\left\lfloor \frac{d-2(j-d+2) +1}{6} \right \rfloor 
= \left \lfloor  \frac{3d-2j-3}{6} \right \rfloor, \\
& \text{sign}(d,j) = \text{triv}(d,-4l+2, j-6l+3) = \left \lfloor \frac{d-4l+1}{6} \right\rfloor +1 \\
& \ \ = \left\lfloor  \frac{3d-2j-3}{6} \right \rfloor +1.
\end{aligned}
\end{equation*}

\vskip 2mm 

{\it Case 2}: \, If $d$ is odd and $j-d$ is odd, then 
\begin{equation*}
\begin{aligned}
& \text{triv}(d,j) = \text{triv}(d-4l, j - 6l) 
=\left\lfloor  \frac{d-4l-1}{6} \right\rfloor +1 \\
& \ \ =\left\lfloor \frac{d-2(j-d+1) -1}{6} \right\rfloor +1
= \left\lfloor  \frac{3d-2j-3}{6}\right \rfloor +1 \\
& \text{sign}(d,j) = \text{sign}(d,-4l, j-6l) = \left\lfloor \frac{d-4l-1}{6} \right\rfloor  \\
& \ \ = \left\lfloor  \frac{3d-2j-3}{6} \right\rfloor.
\end{aligned}
\end{equation*}

\vskip 2mm 

{\it Case 3}: \, If $d$ is even and $j-d$ is even, then 
\begin{equation*}
\begin{aligned}
& \text{triv}(d,j) = \text{sign}(d,j) = \text{sign}(d-4l+2, j-6l+3) \\
& \ \ = \left\lfloor \frac{d-4l+4}{6} \right \rfloor = \left \lfloor \frac{d-2(j-d+2) +4}{6} \right\rfloor
= \left\lfloor  \frac{3d-2j}{6}\right\rfloor. 
\end{aligned}
\end{equation*}

\vskip 2mm 

{\it Case 4}: \, If $d$ is even and $j-d$ is odd, then 
\begin{equation*}
\begin{aligned}
& \text{triv}(d,j) = \text{sign}(d,j) = \text{sign}(d-4l, j-6l) \\
& \ \ = \left\lfloor \frac{d-4l+2}{6} \right \rfloor = \left \lfloor \frac{d-2(j-d+1) +2}{6} \right\rfloor
= \left\lfloor  \frac{3d-2j}{6}\right \rfloor. 
\end{aligned}
\end{equation*}

\vskip 2mm

To summarize, we obtain the following closed form formulas.

\vskip 2mm 

\begin{theorem} \label{j ge d}

{\rm
Suppose that $j \ge d$.

\vskip 2mm 

(a)  If $d$ is odd and $j-d$ is even, then we have 
\begin{equation*}
\text{triv}(d,j) = \left \lfloor \frac{3d-2j-3}{6} \right\rfloor, \quad \text{sign}(d,j) = \left\lfloor \frac{3d-2j-3}{6} \right\rfloor +1.  
\end{equation*}

\vskip 2mm

(b) If $d$ is odd and $j-d$ is odd, then we have 
\begin{equation*}
\text{triv}(d,j) = \left\lfloor \frac{3d-2j-3}{6} \right \rfloor +1, \quad \text{sign}(d,j) = \left \lfloor \frac{3d-2j-3}{6} \right \rfloor. 
\end{equation*}

\vskip 2mm 

(c) If $d$ is even, then we have 
\begin{equation*}
\text{triv}(d,j) = \text{sign}(d,j) = \left\lfloor \frac{3d-2j}{6} \right\rfloor.  
\end{equation*}
}
\end{theorem}

\vskip 2mm 

\begin{example} \ By Theorem \ref{j ge d}, we have 
\begin{equation*}
\begin{aligned}
& \text{triv}(9,12) = \left\lfloor \frac{27-24-3}{6} \right\rfloor + 1 = 1, \ \ 
\text{sign}(9,12) =  \left\lfloor \frac{27-24-3}{6} \right\rfloor = 0, \\
& \text{triv}(10,12) =\text{sign}(10, 12) = \left\lfloor \frac{30-24}{6} \right\rfloor = 1.
\end{aligned}
\end{equation*}

\end{example}

\vskip 3mm

Next, we will consider the case $j \le d-2$. By the recursive relations \eqref{eq:interlocking}, we have
\begin{align*}
& \text{triv}(d, j) = \text{sign}(d+2, j+3)  \\
&\ \ = \text{triv}(d+4, j+6) -1  = \text{sign}(d+6,j+9) -1  \\
& \ \ = \cdots = \text{sign}(d+4-+2, j+6l-3) - l +1= \text{triv}(d+4l, j+6l) -l = \cdots, \\
& \text{sign}(d, j) = \text{triv}(d+2, j+3 -1 \\
&\ \  = \text{sign}(d+4, j+6) -1= \text{triv}(d+6,j+9) -2  \\
& \ \ = \cdots = \text{triv}(d+4l-2, d+6l-3) -l  = \text{sign}(d+4l, j+6l) -l  = \cdots.  
\end{align*}

\vskip 2mm 

If $d+4l-3=j+6l-3$, then $d-j = 2l$ is even and $l = \dfrac{d-j}{2}$. 

If $d+4l-1 = j+6l$, then $d-j = 2l+1$ is odd and $l = \dfrac{d-j-1}{2}$. 

We will compute $\text{triv}(d,j)$ and $\text{sign}(d,j)$ for $j \le d-2$. 

\vskip 3mm 

{\it Case 1}: \,  If $d$ is odd and $d-j$ is even, then we have 
\begin{equation*}
\begin{aligned}
& \text{triv}(d,j) = \text{sign}(d+4l-2, j + 6l-3) -l +1 
=\left \lfloor  \frac{d+4l-3}{6} \right\rfloor -l +1\\
& \ \ =\left\lfloor \frac{d+2(d-j)+3}{6} \right\rfloor -\frac{d-j}{2} + 1 
= \left\lfloor  \frac{3d-2j-3}{6} \right \rfloor - \left(\frac{d-j-2}{2} \right), \\
& \text{sign}(d,j) = \text{triv}(d,+4l-2, j+6l-3) -l =\left(\left \lfloor \frac{d+4l-3}{6} \right\rfloor +1 \right) - \frac{d-j}{2} \\
& \ \ = \left\lfloor  \frac{3d-2j-3}{6} \right\rfloor -\left(\frac{d-j-2}{2} \right).
\end{aligned}
\end{equation*}

\vskip 2mm 

{\it Case 2}: \, If $d$ is odd and $d-j$ is odd, then we have 
\begin{align*}
& \text{triv}(d,j) = \text{triv}(d+4l, j + 6l) -l  
=\left \lfloor  \frac{d+4l-1}{6} \right\rfloor +1 -l \\
& \ \ =\left \lfloor \frac{d+2(d-j +1)-1}{6} \right \rfloor +1 -\frac{d-j+1}{2}  
=\left \lfloor  \frac{3d-2j-3}{6} \right \rfloor - \left(\frac{d-j-3}{2} \right), \\
& \text{sign}(d,j) = \text{triv}(d,+4l, j+6l) - l =\left( \left\lfloor \frac{d+4l-1}{6} \right \rfloor  \right) - \frac{d-j-1}{2} \\
& \ \ = \left \lfloor  \frac{3d-2j-3}{6} \right \rfloor -\left(\frac{d-j-1}{2} \right).
\end{align*}

\vskip 2mm 

{\it Case 3}: \, If $d$ is even and $d-j$ is even, then we have 
\begin{align*}
& \text{triv}(d,j) = \text{sign}(d+4l-2, j + 6l-3) -l +1 
=\left \lfloor  \frac{d+4l}{6} \right \rfloor -l +1\\
& \ \ =\left \lfloor \frac{d+2(d-j)}{6} \right \rfloor -\frac{d-j}{2} + 1 
= \left \lfloor  \frac{3d-2j}{6} \right \rfloor - \left(\frac{d-j-2}{2} \right), \\
& \text{sign}(d,j) = \text{triv}(d,+4l-2, j+6l-3) -l =\left \lfloor \frac{d+4l}{6} \right\rfloor  - \frac{d-j}{2} \\
& \ \ = \left\lfloor  \frac{3d-2j}{6} \right\rfloor -\left(\frac{d-j}{2} \right).
\end{align*}

\vskip 2mm 

{\it Case 4}: \, If $d$ is even and $d-j$ is odd, then we have 
\begin{align*}
& \text{triv}(d,j) = \text{triv}(d+4l, j + 6l) -l  
=\left\lfloor  \frac{d+4l+2}{6} \right\rfloor - l \\
& \ \ =\left\lfloor \frac{d+2(d-j -1) +2}{6} \right \rfloor -\frac{d-j+1}{2}  
= \left \lfloor  \frac{3d-2j}{6} \right\rfloor - \left(\frac{d-j-1}{2} \right), \\
& \text{sign}(d,j) = \text{sign}(d,+4l, j+6l) - l =\left \lfloor \frac{d+4l+2}{6} \right \rfloor  - \frac{d-j-1}{2} \\
& \ \ = \left \lfloor  \frac{3d-2j}{6} \right \rfloor -\left(\frac{d-j-1}{2} \right).
\end{align*}

\vskip 2mm

Moreover, note that the case $j =d-1$ is included in (Case 2) and (Case 4). 

\vskip 2mm 

Therefore we obtain:

\vskip 2mm 

\begin{theorem} \label{thm: j le d-1}

{\rm 
Suppose $j \le d-1$.

\vskip 2mm 

(a) If $d$ is odd and $d-j$ is even, then we have 
\begin{equation*}
\text{triv}(d,j) = \text{sign}(d,j) = \left \lfloor \frac{3d-2j-3}{6} \right \rfloor - \left( \frac{d-j-2}{2} \right). 
\end{equation*}

\vskip 2mm 

(b) If $d$ is odd and $d-j$ is odd, then we have 
\begin{equation*}
\text{triv}(d,j) = \left\lfloor \frac{3d-2j-3}{6} \right\rfloor - \left( \frac{d-j-3}{2} \right), \ \ 
\text{sign}(d,j) = \left\lfloor \frac{3d-2j-3}{6} \right \rfloor - \left( \frac{d-j-1}{2} \right).
\end{equation*}

\vskip 2mm 

(c) If $d$ is even and $d-j$ is even, then we have 
\begin{equation*}
\text{triv}(d,j) = \left\lfloor \frac{3d-2j}{6} \right\rfloor - \left( \frac{d-j-2}{2} \right), \ \ 
\text{sign}(d,j) = \left\lfloor \frac{3d-2j}{6} \right\rfloor - \left( \frac{d-j}{2} \right).
\end{equation*}

\vskip 2mm 

(d) if $d$ is even and $d-j$ is odd, then we have 
\begin{equation*}
\text{triv}(d,j) = \text{sign}(d,j) = \left\lfloor \frac{3d-2j}{6} \right \rfloor - \left( \frac{d-j-1}{2} \right). 
\end{equation*}
}
\end{theorem}

\vskip 2mm

\begin{example} \ By Theorem \ref{thm: j le d-1}, we have  
\begin{equation*}
\begin{aligned}
& \text{triv} (9,6) = \left \lfloor \frac{27-12-3}{6} \right \rfloor - \left( \frac{9-6-3}{2} \right) = 2=0=2, \\
& \text{sign}(9,6) = \left \lfloor \frac{27-12-3}{6} \right \rfloor - \left(\frac{9-6-1}{2}\right) = 2-1 =1, \\
& \text{triv} (10, 6) = \left \lfloor \frac{30-12}{6} \right \rfloor - \left(\frac{10-6-2}{2} \right) = 3-1=2, \\
& \text{sign} (10, 6) = \left \lfloor \frac{30-12}{6} \right\rfloor -\left( \frac{10-6}{2} \right) = 3-2 =1. 
\end{aligned}
\end{equation*}
\end{example}

\vskip 3mm 

We will now move on to standard representations. Assume first that $j=d-1$. 
By the resursive formula \eqref{eq:st-recursive}, we have 
\begin{equation*}
\begin{aligned}
 & \text{st}(d,d-1) = \text{st}(d-2,d-4) +1 = \text{st}(d-6, d-7) +2 \\
 & \ \  = \cdots = \text{st}(d-2l, d-3l-1) +l. 
 \end{aligned}
\end{equation*}

If $d=3l$, then 
$$\text{st}(d, d-1) = \text{st}(l, -1) + l = l =\frac{d}{3} = \left \lfloor \frac{d}{3} \right\rfloor 
= \left \lfloor \frac{d+1}{3} \right \rfloor.$$

If $d=3l+1$, then 
$$\text{st}(d, d-1) = \text{st}(l+1, 0) + l = l = \frac{d-1}{3} =\left \lfloor \frac{d-1}{3} \right \rfloor
=\left\lfloor  \frac{d+1}{3} \right \rfloor.$$ 

If $d = 3l +2$, then 
$$\text{st}(d, d-1) = \text{st}(l+2, 1) = l+1= \frac{d-2}{3} + 1 = \frac{d+1}{3} 
=\left \lfloor \frac{d+1}{3} \right\rfloor.$$

Therefore we obtain 
\begin{equation} \label{eq:st=d-1}
\text{st} (d, d-1) =\left \lfloor \frac{d+1}{3} \right\rfloor.
\end{equation}

\vskip 2mm 

Note that the recursive relations \eqref{eq:st-recursive} can be rephrased as 
\begin{equation} \label{eq:st-recursive2}
\text{st}(d,j) = \begin{cases}
\text{st}(d+2, j+3) -1 \ \ & \text{if} \ j \le d-2, \\
\text{st}(d+2, j+3) \ \ & \text{if} \ j \ge d-1.
\end{cases}
\end{equation}

\vskip 2mm 

Hence if $j \ge d$, we have 
\begin{equation*}
\text{st}(d,j) = \text{st}(d-2, j-3) = \text{st}(d-4, j-6) = \cdots 
= \text{st}(d-2l, j-3l) = \cdots 
\end{equation*}
Thus when we reach the point where $d-2l-1 = j - 3l$.  $l=j-d+1$ and we have 
$$\text{st}(d,j) = \text{st}(d-2l, j-3l) = \left\lfloor \frac{d-2l+1}{3} \right\rfloor 
=\left \lfloor \frac{3d-2j-1}{3} \right \rfloor.$$

\vskip 2mm 

If $j \le d-2$, then 
$$\text{st} (d,j) = \text{st}(d+2, j+3) -1 = \text{st}(d+4, j+6)-2 
= \cdots = \text{st}(d + 2l, j+3l) -l = \cdots $$.
Thus when $d + 2l -1 = j+3l$, $l=d-j-1$ and we have 
\begin{equation*}
\begin{aligned}
& \text{st}(d,j) = \text{st}(d+2l, j+3l) -l = \left \lfloor \frac{d+2l +1}{3} \right\rfloor -l \\
& \ \ = \left\lfloor \dfrac{d+2(d-j-1)+1}{3} \right \rfloor -(d-j-1) \\
& \ \ = \left \lfloor \frac{3d-2j-1}{3} \right \rfloor - (d-j-1).
\end{aligned}
\end{equation*}

Note that the case $j=d-1$ is included in the above formula. 

\vskip 2mm 

Therefore, the multiplicities of standard representations are given by the 
following theorem. 

\vskip 2mm 

\begin{theorem} \label{thm:st-closed form}
{\rm 
For all $d \ge 3$ and $j \ge 0$, we have 
\begin{equation} \label{eq:st-closed form}
\text{st}(d,j) = \begin{cases}
 \left\lfloor \dfrac{3d-2j-1}{3} \right\rfloor \ \ & \text{if} \ j \ge d, \\ \vspace{-0.4cm}
 \\
 \left\lfloor \dfrac{3d-2j-1}{3}  \right\rfloor - (d-j-1) \ \ & \text{if} \ j \le d-1. 
\end{cases}
\end{equation}
}
\end{theorem}

\vskip 2mm 

\begin{example} \ By \eqref{eq:st=d-1} and Theorem \ref{thm:st-closed form}, we have 
\begin{align*}
& \text{st}(9,8) = \left \lfloor \frac{9+1}{3} \right \rfloor=3, \ \ 
\text{st}(9,12) = \left \lfloor \frac{27-24-1}{3} \right \rfloor = 0,\\
& \text{st}(9,6) = \left \lfloor \frac{27-12-1}{3}\right \rfloor - (9-6-1) = 4-2=2, \\
& \text{st}(10, 9) = \left \lfloor   \frac{10+1}{3} \right \rfloor =3, \ \ 
\text{st}(10, 12) = \left \lfloor  \frac{30 - 24 -1}{3} \right \rfloor =1,  \\
& \text{st}(10, 6) = \left \lfloor \frac{30 - 12 -1}{3}  \right \rfloor  - (10-6-1) = 5-3 =2. 
\end{align*}
\end{example}

Note that all the computations coincide with each other both in recursive form and in closed form.

\end{document}